\documentclass[12pt,a4paper,twoside]{article}

\usepackage[utf8]{inputenc}
\usepackage{amssymb,amsthm}
\usepackage[english]{babel}
\usepackage[leqno]{amsmath} 
\usepackage{mathtools}
\usepackage{bbm}
\usepackage[twoside, bindingoffset=5mm, left=15mm, right=25mm, top=28mm, bottom=28mm,a4paper]{geometry} 
\usepackage{float}
\usepackage{fancyhdr}
\usepackage[breaklinks,plainpages=false,pdfpagelabels=true]{hyperref}
\usepackage{bookmark}
\usepackage{tikz}
\usepackage{enumerate}
\usepackage{amsmath}
\usepackage{newpxtext}
\usepackage{keyval}
\usepackage{constants}
\usepackage{sectsty}

\allowdisplaybreaks 

\newcommand{\R}{\mathbbm {R}}
\newcommand{\M}{\mathcal {M}}
\newcommand{\N}{\mathbbm {N}}

\newcommand{\W}{\mathcal {W}}
\newcommand{\Kk}{\mathcal {K}}
\newcommand{\diff}{\operatorname{d}\!}

\newcommand{\dist}{\operatorname{dist}}
\newcommand{\meanint}{{\int{\mkern-19mu}-}}

\renewcommand{\div}{\operatorname{div}}

\newcommand {\cb}[1] {\tikz[baseline=(char.base)]{
            \node[shape=circle,draw,inner sep=2pt] (char) {\itshape{\textbf{#1}}};}}
\newcommand {\cbm}[1] {\tikz[baseline=(char.base)]{
            \node[shape=circle,draw,inner sep=2pt] (char) {\scriptsize\itshape{\textbf{#1}}};}}

\newtheoremstyle{plain}	
	{1em}{1em}	
	{\itshape}	
	{}	
	{\bfseries}	
	{}{\newline}	
	{\thmnumber{#2 }\thmname{#1}\thmnote{ (#3)}}

\newtheorem{thm}{Theorem}

\newtheorem{lemma}[thm]{Lemma}
\newtheorem{Proposition}[thm]{Proposition}
\newtheorem{remark}[thm]{Remark}

\title{Graphical Willmore Problems with Low-Regularity Boundary and Dirichlet Data}
\author{Boris Gulyak\\
    \small  Otto-von-Guericke-Universität,\\
   \small  Fakultät für Mathematik,\\
   \small   Postfach 4120,\\
   \small  39016 Magdeburg,\\
   \small Germany\\
    \small \texttt{bgulyak@ovgu.de}}

\begin{document}
\maketitle

\abstract{We establish existence and regularity results for boundary value problems arising from the first variation of the Willmore energy in the graphical setting. Our focus lies on two-dimensional surfaces with fixed clamped boundary conditions, embedded in three-dimensional Euclidean space, and represented as graphs of height functions over domains with non-smooth boundaries.
Our approach involves constructing solutions through linearization and a fixed-point argument, requiring small boundary data in suitable functional spaces.  Building on the results of Koch and Lamm \cite{koch2012geometric}, we rewrite the Willmore equation for graphs in a divergence form that allows the application of weighted second-order Sobolev spaces.  This reformulation significantly weakens the regularity assumptions on both the boundary and the Dirichlet data, reducing them to the $C^{1+\alpha}$-class, while the solution remains smooth in the interior.
Moreover, we extend the existence theory to domains with merely Lipschitz boundaries within a purely weighted Sobolev framework. Our approach is also applicable to other higher-order geometric PDEs, including the graphical Helfrich and surface diffusion equations.
}





\maketitle

\section{Introduction}

We consider sufficiently smooth two-dimensional surfaces $S\subset\R^3$. The Willmore energy of such a surface is defined by
\begin{align}
    \W(S):= \frac{1}{4}\int_S H^2\diff S \label{eq:meinewilldef},
\end{align}
 where $H=\kappa_1+\kappa_2$ denotes the mean curvature, given as the sum of the principal curvatures. If \( H > 0 \) at some point \( p_0 \) on the surface \( S \), we interpret the surface as locally bending, on average, in the direction of the unit normal vector \( N(p_0) \) of \( S \) at \( p_0 \). Originally introduced by Germain and Poisson  \cite{germain1821recherches} in the early 19th century, the Willmore energy re-emerged in the context of conformal geometry in the early 20th century through the work of Thomsen \cite{thomsen1924grundlagen}, and was later popularized by Willmore. Beyond its geometric relevance, the Willmore energy also plays a central role in modeling thin elastic plates that resist bending, a topic first investigated by Germain and Poisson \cite{germain1821recherches}. In addition, it forms the basis of modern theories describing biological membranes as lipid bilayers, as introduced by Helfrich and Canham \cite{helfrich1973elastic, canham1970minimum}.

One of the defining geometric properties of the Willmore functional is the conformal invariance of the  following quantity
\[
\int_S \left( \frac{1}{4} H^2 - \mathcal{K} \right) \, \mathrm{d}S,
\]
where \( \mathcal{K} = \kappa_1 \cdot \kappa_2 \) denotes the Gaussian curvature. This expression remains invariant under conformal transformations \( \Phi\colon \mathbb{R}^3 \to \mathbb{R}^3 \), including Möbius transformations such as scaling, rotations, and inversions; see \cite{willmore1996riemannian}. 

For a sufficiently smooth surface \( S \), the Willmore equation is derived by setting the first variation of the Willmore energy to zero, as discussed in \cite[p.7, Remark 2.3, (2.1)]{dall2012uniqueness}:
\begin{align}
    \Delta_S H + 2 H\left( \frac{1}{4} H^2 - \mathcal{K} \right) = 0 \quad \text{on } S, \label{eq:willmoregen}
\end{align}
where \( \Delta_S \) denotes the Laplace-Beltrami operator on \( S \). Surfaces satisfying equation \eqref{eq:willmoregen} are referred to as Willmore surfaces. A central characteristic of this equation is its quasilinear and fourth-order elliptic nature; although it is elliptic, it is not uniformly elliptic.  In particular, significant variations in the tangent planes can cause a strong degeneration of ellipticity. Willmore surfaces also generalize the concept of minimal surfaces, since minimal surfaces satisfy \( H = 0 \) and thus trivially minimize the Willmore energy.  

While there exists a broad range of results on Willmore surfaces without boundary (i.e., closed surfaces), for references see \cite{willmore1996riemannian, li1982new, marques2014willmore, simon1993existence, bauer2003existence, simonett2001willmore, kuwert2004removability, chill2009willmore, blatt2009singular, koch2012geometric}, the corresponding boundary value problems and the analysis of Willmore flows are significantly more challenging and less well understood. One major difficulty lies in the fact that standard scaling arguments do not apply in the presence of boundaries. Moreover, there are typically no a-priori bounds for solutions to the Willmore equation, nor for minimizers or minimizing sequences of the Willmore energy.  In general, due to the strong nonlinearity of the Willmore equation, uniqueness of solutions cannot be expected, as emphasized in  \cite{eichmann2016nonuniqueness}. 

Various types of boundary conditions have been studied; see the surveys  \cite{nitsche1993boundary, grunau2018boundary}.  In the case of a clamped boundary condition, also called the Dirichlet condition, where both the position of the boundary in space and the tangential planes along the boundary are prescribed, Schätzle \cite{schatzle2010willmore} established the existence of branched Willmore minimizers in $S^3\equiv \R^3\cup \{\infty\}$ heavily relying on methods from geometric measure theory, such as varifold techniques. In general, these minimizers are not graphical, may fail to be embedded, and can even include the point at infinity. Moreover, no a priori bounds are available. Only under additional assumptions, such as projectivity (see \cite{deckelnick2017minimising}) or an energy bound $\W(S)<4\pi$, do area and diameter become controlled in terms of the Willmore energy and the boundary length, as shown in \cite{pozzetta2018plateau} and \cite[Subsection 4.1]{gulyak2024boundary}. Furthermore, by imposing an area constraint, Da Lio, Palmurella, and Rivière \cite{da2020resolution} proved an existence result using classical parametric techniques.

In order to reduce complexity and gain insights into the phenomena described by the Willmore equation, one can impose additional geometric constraints on the surface. This simplification makes the problem more tractable and allows for the extraction of geometric and analytical properties of the Willmore energy. For instance, by assuming axial symmetry, one obtains surfaces of revolution, for which various results exist; see \cite{dall2008classical,  dall2013unstable, deckelnick2007boundary, grunau2023willmore} and the recent contribution \cite{schlierf2024convergence}. Alternatively, if a surface is assumed to be translationally invariant in a fixed direction, it can be described by a curve with one-dimensional Willmore energy, commonly referred to as the elastic energy; see \cite{deckelnick2007boundary,deckelnick2009stability}.

\subsection{Graphical setting}
In our setting, we consider only surfaces that are projectable onto a sufficiently smooth, bounded domain $\Omega\subset\R^2$. The height of each surface $S$ with respect to this projection is described by a sufficiently smooth function \( u\colon \overline{\Omega} \to \mathbb{R} \). The surface \( S \) is then parameterized by \( \overline{\Omega} \ni (x_1, x_2) \mapsto (x_1, x_2, u(x))^T \in \R^3\) and we denote this surface as \( S = \Gamma(u) \) referring to it as the graph of \( u \). In this representation, the Willmore equation \eqref{eq:willmoregen} transforms into a nonlinear biharmonic equation of the form
\begin{align}
    \Delta^2 u = F(D^4u, D^3u, D^2u, \nabla u) \quad \text{in } \Omega, \label{eq:firstrep}
\end{align}
where \( \Delta^2 = \Delta \Delta \) denotes the biharmonic operator, and \( F \) is a nonlinear polynomial in its arguments  \( D^4u \), \( D^3u \), \( D^2u \), \( \nabla u \), and \( \sqrt{1 + |\nabla u|^2}^{-1} \). Importantly, all non-vanishing monomials in $F$ are at least cubic, and the highest-order derivatives \( D^4u \) and \( D^3u \) appear linear in corresponding monomials. The appearance of $D^4 u$ on the right side suggests using Sobolev spaces \( W^{4,p}(\Omega) \) or Hölder spaces \( C^{4+\alpha}(\overline\Omega) \), both of which accommodate fourth-order derivatives and are suitable for establishing existence results. Moreover, there is a divergence structure as described in \cite{koch2012geometric} for Willmore flow of entire graphs, so it is possible to work in terms of second-order solution spaces, thus avoiding the need to work directly in fourth-order settings.

To implement the clamped boundary condition for graphical Willmore surfaces, we impose Dirichlet boundary conditions. In this work, we aim to solve the following elliptic boundary value problem for the Willmore equation, where both the function and its normal derivative are prescribed along the boundary \( \partial \Omega \):
\begin{align}
    \left\{
    \begin{aligned}
    \Delta_{\Gamma(u)} H + 2 H \left( \frac{1}{4} H^2 - \mathcal{K} \right) &= 0 \quad \text{in } \Omega, \\
    u = g_0, \quad \frac{\partial u}{\partial \nu} &= g_1 \quad \text{on } \partial \Omega,
    \end{aligned}\right. \tag{W}\label{eq:dirwillmore}
\end{align}
where \( g_0 \) and \( g_1 \) are sufficiently smooth functions defined on \( \partial \Omega \).

Next, we briefly discuss the biharmonic part of the equation. As emphasized in \cite{grunau2018boundary}, it is important to contrast general fourth-order operators with well-studied second-order elliptic operators, such as the Laplace operator, for which established tools like maximum principles, comparison principles, and Harnack inequalities are readily available. For fourth-order or higher-order operators, however, such tools are either unavailable or significantly more difficult to apply, thereby increasing analytical complexity.

Nevertheless, in the case of Lipschitz boundary domains, a generalization of Miranda’s maximum modulus theorem for the biharmonic operator $\Delta^2$ was proven by Pipher and Verchota in \cite{Pipher1993}. It states that if $\Omega\subset \R^2$ is a bounded Lipschitz graph domain and $u$ is the unique $L^2$-solution to $\Delta^2 u=0$ with  $|\nabla u| \in L^\infty(\partial\Omega)$, then it holds
\begin{align*}
    \sup_{x\in \Omega} \big| \nabla  u (x) \big| \le \Cl{miranda1} \| \nabla  u \|_{L^\infty(\partial\Omega)},
\end{align*}
with $\Cr{miranda1}$ depending only on the Lipschitz structure of $\partial\Omega$.

In the case of the graphical Willmore equation, however, no comparable maximum modulus result can be expected for large boundary data. Under clamped boundary conditions, if the slope of $u$ on the boundary is sufficiently steep, a surface minimizing the Willmore energy may bend over itself and form three layers over some region in $\Omega$ instead of remaining single-valued. As a consequence, the projectability of the surface may be lost; see the stadium example in \cite[p. 1690 and Fig. 8]{barrett2017}.

\subsection{Main Results and Ideas}
Due to the nonlinearity of the Willmore equation, it is to be expected that some form of smallness condition on the boundary data norm is required for any existence result. In the parabolic setting, for instance, Koch and Lamm \cite{koch2012geometric} demonstrated the existence of the Willmore flow for entire graphs under the assumption of initial data with a small Lipschitz norm. Additionally motivated by the observation from the previous subsection namely that, similar to the quantity \(
\int_S \left({H^2}/{4} - \mathcal{K}\right) \, \mathrm{d}S
\), the gradient  $\nabla u$ is invariant to rescaling of the exterior space $ \R^3$, our initial aim, for a smooth $\partial\Omega$, was to establish the following:

\emph{\textbf{Goal:} For clamped boundary conditions, if the slope $\| \nabla u \|_{L^\infty(\partial\Omega)}$  on the boundary is sufficiently small (depending on $\Omega$), then there exists a graphical solution to the Willmore equation satisfying the prescribed boundary data.}

In this paper, we establish a corresponding existence result under a $C^{1+\alpha}$-smallness condition on the boundary data, for arbitrarily small $\alpha\in(0,1)$. Moreover, in a more specialized setting, we are able to relax the regularity assumptions on $\partial\Omega$ to the mere Lipschitz class and still derive existence results. We also note that the main parts of this work have already appeared in Section 5 of the author's PhD thesis \cite{gulyak2024boundary}.

The existence results presented in this paper rely on  of weighted second-order Sobolev spaces $W^{2,a}_p(\Omega)$ for treating the Willmore boundary value problem, where the weight is given by the distance to the boundary. With a suitable choice of parameters, these spaces embed into $C^1(\overline{\Omega})$, see Lemma~\ref{lem:cemb} (c). This embedding implies a bounded gradient norm $\|\nabla u\|_{L^\infty(\Omega)}$,  which is crucial for the linearization estimates. 
One drawback of the weighted Sobolev setting is the technically involved trace theory. For this reason, we first formulate our results using Hölder continuous boundary data and present them in a simplified form.

\begin{thm}
   \label{thm:weightedmainhoelder}
Let $\Omega\subset\R^2$ be a bounded domain with $\partial\Omega\in C^{1+\alpha}$ for some $\alpha\in(0,1)$. Assume that $\beta\in(0,\alpha), g_0\in C^{1+\alpha}(\partial\Omega)$ and $g_1\in C^{\alpha}(\partial\Omega)$. Additionally, we suppose that $\|g_0'\|_{C^{\alpha}(\partial\Omega)}+ \|g_1\|_{C^{\alpha}(\partial\Omega)} <K$ for some $K>0$. 

Then there exists a constant $\delta=\delta(\alpha,\beta, K, \Omega)>0$ such that if $\|g_0'\|_{C^0(\partial\Omega)}+\|g_1\|_{C^{0}(\partial\Omega)}<\delta$, then there exists a solution $u\in C^{1+\beta}(\overline\Omega)\cap C^\infty(\Omega)$ to the Willmore Dirichlet problem \eqref{eq:dirwillmore}.
\end{thm}

The proof is provided in Subsection~\ref{subsec:hoelderboundary}.
 Compared to the \( C^{4+\alpha}(\partial{\Omega}) \)-smallness condition required by Nitsche \cite{nitsche1993boundary}, who approached the problem by decomposing it into a coupled system of second-order elliptic equations for $u$ and $H$, complemented by coupled Dirichlet boundary data, our method constitutes a substantial improvement.  In particular, we achieve a significant reduction in the regularity requirements for the boundary \( \partial\Omega \). Moreover, in the setting of weighted Sobolev spaces, we establish an existence result even when $\partial\Omega$ is merely Lipschitz continuous. This, however, comes at the price of more restrictive conditions on the Dirichlet boundary data. The corresponding result is presented below as the second main theorem in a more technical formulation.
\begin{thm}
 \label{thm:lipschitz}
Assume that $\Omega\subset \R^2$ is a bounded Lipschitz domain whose Lipschitz constant
does not exceed $M$. Then there exist a constant $\Cl{2vmobound}=\Cr{2vmobound}(M)>0$ such that if
\begin{itemize}
    \item $p\in(2,2+\Cr{2vmobound}), a\in \Big(0,1-\frac{2}{p}\Big)\cap(0,\Cr{2vmobound})$ with $s:=1-a-\frac{1}{p}$,
    \item $g:=\{g_0,g_1\}\in \dot W^{1+s}_p(\partial\Omega)$ such that $g_0, \nu g_1+\nabla_{\rm tan} g_0\in L^{\infty}(\partial\Omega)$,
    \item $\|g\|_{\dot W^{1+s}_p(\partial\Omega)} <K$ for some $K>0$.
\end{itemize}

Then there exists a constant $\delta=\delta(\Omega,K,p,a,M)>0$ such that if 
\begin{align*}
    \|\nu g_1+\nabla_{\rm tan} g_0\|_{L^\infty(\Omega)}<\delta,
\end{align*}
then there exists a variational solution $u\in {W^{2,a}_p(\Omega)}\cap C^\infty(\Omega)$ to the Willmore-type Dirichlet problem, thus $u$ solves \eqref{eq:Willmoreellipt} with the right-hand side given by \eqref{eq:bi}.
\end{thm}
 Since its formulation is more abstract, we postpone the presentation of the result to the final section, Subsection~\ref{subsec:purelipsch}, where we also refer to the proof provided in Subsection~\ref{subsec:mainweighted}. For the necessary notation, see Subsections~\ref{subsec:weightedsobspace} and \ref{subsec:dirichdt}.

Next, we list the four key tools employed in this work:
\begin{itemize}
    \item[\cb 1] \textit{Rewriting the graphical Willmore equation with divergence right-hand side}, following the approach of Koch and Lamm \cite{koch2012geometric}, as carried out in Lemma~\ref{lem:willrewriten}. There we multiply the geometric Willmore equation \eqref{eq:willmoregen} by $Q=\sqrt{1+|\nabla u|^2}$ and using Einstein summation notation, obtain the semilinear divergence form:
\begin{align}
\Delta^2 u = D_{i} b^{i}_1[u] + D^2_{ij} b^{ij}_2[u]   \quad \text{ in }\Omega,
\label{eq:divstr}
\end{align}
where the terms $b_i[u]$ are polynomials in $D^2u, \nabla u$ and $Q^{-1}$, consisting only of monomials of degree at least three. Specifically, $b_1[u]$ is quadratic in $D^2u$ and $b_2[u]$ is linear in $D^2u$ with the following bounds:
\begin{align*}
   \big |b_1[u]\big|\le C|\nabla u| \cdot|D^2 u|^2, \qquad \big|b_2[u]\big| \le C |\nabla u|^2 \cdot |D^2u|,
\end{align*}
 with some algebraic constant $C>0$.

  \item[\cb 2] \emph{Weighted Sobolev spaces}, together with existence results and estimates due to Maz’ya, Shaposhnikova, and Mitrea in \cite{maz2011recent, maz2010dirichlet}, allow for the treatment of biharmonic operators with divergence-form right-hand sides, as stated in Proposition~\ref{weightschauderlpdiv}. This generalization of classical (unweighted) Sobolev estimates permits weaker assumptions on both the boundary and the boundary data than those required in the unweighted setting.
  \item [\cb 3] \emph{Weighted Sobolev embeddings} due to Opic and Hardy \cite{opic1990hardy}, as presented in Lemma~\ref{lem:cemb} (a), (b), and (d), play an important role in our analysis. Combined with Hölder's inequality, they allow us to estimate the right-hand side of the equation in the setting of Proposition~\ref{weightschauderlpdiv}, which addresses a general biharmonic problem.
  \item[\cb 4] \emph{$L^\infty$-estimate of the gradient via weighted Sobolev norms}, due to Brown and Opic \cite{brown1992embeddings} and stated in Lemma~\ref{lem:cemb} (c), is of central importance. The structure of the reformulated graphical Willmore equation reveals that a bound on $|\nabla u|_{L^\infty(\Omega)}$ is essential: without such a bound, the surface may lose projectivity, and, as discussed above, a strong degeneration of ellipticity may occur. 
\end{itemize}

At this point, we briefly discuss the advantages of reformulating the graphical Willmore equation. The existence results established in this work rely on a new divergence structure of the equation, derived from the computations in \cite{deckelnick2006error, koch2012geometric}.
This divergence form enables us to formulate the problem in function spaces involving only second-order derivatives, in contrast to the more restrictive fourth-order frameworks used, for example, in \cite{nitsche1993boundary}. This choice is particularly natural, given that the Willmore energy depends solely on second-order derivatives. We also note that alternative divergence-based techniques have been employed in the study of the Willmore equation; see, for instance, \cite{riviere2008analysis}.

Moreover, it is essential to multiply the Willmore equation by $Q$  beforehand. Otherwise, as in \cite[p. 215, Lemma 3.2]{koch2012geometric}, an additional term $b_0[u]$, cubic in $D^2u$, would appear on the right-hand side.This term would necessitate the use of weighted Sobolev spaces ${W^{2,a}_p(\Omega)}$ spaces with $p\ge3$, thereby imposing stricter regularity requirements on the boundary data. In contrast, the absence of the term $b_0[u]$ in our formulation allows us to work with the milder condition $p>2$.

Furthermore, an important advantage of using weighted Sobolev spaces is that they permit working with weaker Hölder regularity on the boundary than in the unweighted Sobolev setting. Specifically, the divergence structure of the reformulated graphical Willmore equation \eqref{eq:divstr} enables the use of Hölder spaces $C^{2+\alpha}(\overline\Omega)$ and Sobolev spaces $W^{2,p}(\Omega)$, as demonstrated in \cite[Subsections 5.2 and 5.3]{gulyak2024boundary}. In the unweighted case, the trace theorem for $W^{2,p}(\Omega)$ with $p>2$ restricts the admissible boundary spaces to $C^{1+\alpha}(\partial\Omega)$ with $\alpha>1/2$. In contrast, in the weighted framework, we can use the boundary spaces $C^{1+\alpha}(\partial\Omega)$ for all $ \alpha\in (0,1)$.

All in all, we emphasize that these results are primarily analytical rather than geometric in nature. As such, they can be generalized to other fourth-order elliptic Dirichlet boundary value problems with a similar structural form. Examples include the graphical Helfrich equation with small spontaneous mean curvature \cite{helfrich1973elastic}, or the graphical surface diffusion equation. For the corresponding parabolic case, see \cite[p. 216, Lemma 3.3]{koch2012geometric}.

In future work, one aim will be to replace the current $C^{1+\alpha}$-smallness on the boundary data by a mere $C^{1}$-condition, which appears to be necessary in order to fully achieve the objective stated at the beginning of this section. Moreover, explicit values for the constants involved in the aforementioned conditions remain unknown, even for simple domains such as $\Omega=\mathbb S^1$.

In the following, we provide a brief outline of the paper:
\begin{enumerate}
    \item[]  First, omitting the more abstract framework of weighted Sobolev spaces, we present in Section~\ref{sec:2}, specifically in Lemma~\ref{lem:willrewriten}, how the Willmore equation can be rewritten as a biharmonic operator with a divergence-form right-hand side. This reformulation constitutes the first new result of this paper.
    \\
    \item[]   In Section~\ref{sec:weightsob}, we compile the necessary notation and results concerning weighted Sobolev spaces as found in the literature. Following the introduction of Lipschitz boundaries, we define the weighted Sobolev spaces in Subsection~\ref{subsec:weightedsobspace}. Subsequently, in Subsection~\ref{subsec:weightedemb}, we present the embedding results used later, particularly in Lemma~\ref{lem:cemb}. In Subsection~\ref{subsec:dirichdt}, we introduce Besov spaces on the boundary along with specific trace spaces tailored to clamped boundary conditions. Finally, in Subsection~\ref{sec:existencereg}, we define the BMO modulo VMO character for boundary regularity and formulate weighted Sobolev estimates and existence results for a general biharmonic operator with divergence-form right-hand side.\\
    \item[] In Section~\ref{sec:existencereg}, we present the main proofs. As a preparatory step, Subsection~\ref{subsec:prelim} provides weighted Sobolev and $L^\infty$-gradient estimates for the general nonhomogeneous biharmonic problem in Proposition~\ref{weightschauderlpdiv}, along with weighted estimates for the right-hand side in Lemma~\ref{lemm:wprem12}. The core result is established in Subsection~\ref{subsec:mainweighted}, where we prove Theorem~\ref{thm:wlpexistweight} using a fixed-point argument in weighted Sobolev spaces. The corresponding result for domains with Lipschitz boundaries is stated in Theorem~\ref{thm:lipschitz} in Subsection~\ref{subsec:purelipsch}. Lastly, in Subsection~\ref{subsec:hoelderboundary}, we treat the case of Hölder boundary data, culminating in Theorem~\ref{thm:weightedmainhoelder}.
\end{enumerate}


\section{Rewriting the Willmore Equation} \label{sec:2}
In this section, we investigate the graphical Willmore equation within the framework of variational solutions. As already noted by Nitsche in \cite{nitsche1993boundary}, the principal part of the Willmore equation essentially corresponds to the biharmonic operator with a nonlinear right-hand side, as becomes evident after linearization. More precisely, based on the work of Dziuk and Deckelnick \cite[(1.5)–(1.9)]{deckelnick2006error} on the Willmore flow of graphs, one observes that in the graphical setting, the Willmore equation is indeed a fourth-order equation, involving derivatives up to fourth order. It can be written as
\begin{align}
  0 =  \Delta_g H + \frac{1}{2}H^3-2H \Kk= \div \left( \frac{1}{Q} \left( \left(I-\frac{\nabla u \otimes \nabla u }{Q^2}
\right) \nabla (Q H) \right) - \frac{H^2 }{2Q} \nabla u\right), \label{eq:willmoreeq}
\end{align}
where the mean curvature is given by
\begin{align*}
    H= \div \left( \frac{\nabla u}{Q} \right) = \frac{\Delta u}{Q} - \frac{\nabla u\cdot ( D^2 u \nabla u)}{Q^3}.
\end{align*}
The central idea is to isolate the biharmonic part $\Delta^2u$ in equation \eqref{eq:willmoreeq} from the remaining terms in a novel way, analogous to the reformulation introduced by Koch and Lamm in \cite[Lemma 3.2, p. 215]{koch2012geometric} in the context of the Willmore flow for entire graphs. This decomposition allows us to apply various existence and regularity results for inhomogeneous biharmonic equations in weaker function spaces.

 One of the key aspects in proving the existence results is that the precise algebraic structure of the nonlinear terms is not our primary concern. To systematically absorb all arising algebraic constants, we adopt a notation already introduced in \cite[p. 215]{koch2012geometric}. Specifically, we use the symbol $\star$ to denote an arbitrary linear combination of contractions between derivatives of $u$. For instance, we write  $|\nabla u|^2 = \nabla u \star \nabla u$ and $\nabla_i u D^2_{ij} u \nabla _j u= \nabla u \star  D^2 u \star \nabla u$. This yields the following representation for the mean curvature:
\begin{align}
    H=   Q^{-1}\star D^2u + Q^{-3}\star\nabla u\star D^2 u \star \nabla u. \label{eq:Hstar}
\end{align}
In addition, we introduce an abstract notation for gradient polynomials:
 \begin{align*}
     P_\ell(\nabla u) = \underbrace{\nabla u\star \dots \star \nabla u }_{\ell \text{-times}},
 \end{align*}
which facilitates the formulation of the non-biharmonic divergence terms in the lemma that follows.
 
 \begin{lemma} \label{lem:willrewriten}
 The graphical Willmore equation \eqref{eq:willmoreeq} can be reformulated in the following divergence form
 \begin{align}
     0= \div \left( \frac{1}{Q} \left( \left(I-\frac{\nabla u \otimes \nabla u }{Q^2}
\right) \nabla (Q H) \right) - \frac{H^2 }{2Q} \nabla u\right)
= \Delta^2 u - D_{i} b^{i}_1[u] - D^2_{ij} b^{ij}_2[u],  \label{eq:Willmoreellipt}
\end{align}
where the lower-order terms are given by
\begin{align}
\begin{aligned}
b_1[u] =&\ D^2 u \star D^2 u \star \sum_{k=1}^3 Q^{-2k-1}P_{2k-1}(\nabla u),   \\
    b_2[u]=&\ D^2 u \star \sum_{k=1}^2 Q^{-2k-1}P_{2k}(\nabla u) + D^2 u \star P_2 (\nabla u) \star (Q(1+Q))^{-1}. 
\end{aligned}
    \label{eq:bi}
\end{align}
 \end{lemma}
\begin{proof}
We reformulate the Willmore equation by analyzing its terms individually. We begin by computing the gradient of the scalar function $Q$, which appears in the first term containing the biharmonic component:
\begin{align*}
    \nabla (Q)= \frac{D^2u \, \nabla u}{Q}.
\end{align*}
This identity is needed to handle the divergence term involving $QH$
\begin{align*}
    \div\left( \frac{1}{Q} \nabla (Q H) \right) 
    =&\  \div\left(  \nabla ( H) + \frac{D^2u \nabla u}{Q^3} QH \right) \\
    = &\ \Delta \left( \frac{\Delta u }{Q} \right) - \Delta \left( \frac{\nabla u (D^2 u \nabla u)}{Q^3} \right)  +\div\left( \frac{D^2u \nabla u}{Q^3} QH  \right).
\end{align*}
Especially since $(1-Q)(1+Q)=1-Q^2 = 1-1-|\nabla u|^2= -|\nabla u|^2$, it follows
\begin{align*}
\Delta \left( \frac{\Delta u }{Q} \right)   =& \Delta \left( \Delta u+\left(\frac{1}{Q}-1\right)\Delta u \right) = 
\Delta ^2 u +\Delta \left(\frac{1-Q}{Q}\Delta u \right) =
\Delta ^2 u - \Delta \left( \frac{|\nabla u|^2}{Q(1+Q)}\Delta u  \right).
\end{align*}
Combining all terms, we arrive at the identity
\begin{align*}
        \div\left( \frac{1}{Q} \nabla (Q H) \right) 
    = \Delta ^2 u - \Delta \left( \frac{|\nabla u|^2}{Q(1+Q)}\Delta u  \right) - \Delta \left( \frac{\nabla u (D^2 u \nabla u)}{Q^3} \right)  +\div\left( \frac{D^2u \nabla u}{Q^3} QH  \right).
\end{align*}
The second term in \eqref{eq:willmoreeq} can be rewritten using Einstein summation notation. Specifically, we consider
\begin{align*}
 \nabla_i\Bigg( & \frac{\nabla_i u \nabla _j u}{Q^3}  \nabla_j (Q H) \Bigg) 
    = \nabla_i\left( \nabla_j \left( \frac{\nabla_i u \nabla _j u}{Q^3}Q H \right)
     -\nabla_j\left( \frac{\nabla_i u \nabla _j u}{Q^3} \right)  Q H 
     \right)  \\
     =&\  \nabla_i\left( \nabla_j \left( \frac{\nabla_i u \nabla _j u}{Q^3}Q H \right)
     -(-3)\nabla_j (Q) \frac{\nabla_i u \nabla _j u}{Q^4}   Q H 
     - \frac{\nabla_j(\nabla_i u \nabla _j u)}{Q^3}   Q H 
     \right)  \\
     =&\  \nabla_i\left( \nabla_j \left( \frac{\nabla_i u \nabla _j u}{Q^3}Q H \right)
     + 3  \frac{\nabla_i u \nabla _j u (D^2u \nabla u)_j }{Q^5}Q H  -   \frac{D^2_{ij} u \nabla _j u +D^2_{jj} u \nabla _i u }{Q^3}Q H 
     \right)  \\
    =&\ D_{ij}^2 \left( \frac{\nabla_i u \nabla _j u}{Q^2} H \right)
     +  \div  \left( 3 \frac{\nabla u\cdot ( D^2 u \nabla u) }{Q^4}H \nabla u - \frac{H}{Q^2}  D^2u \nabla u  -   \frac{ \Delta u H  }{Q^2}   \nabla  u
     \right) \\
     =&\ D_{ij}^2 \left( \frac{\nabla_i u \nabla _j u}{Q^2} H \right)
     +  \div  \left( -3\frac{H^2 }{Q} \nabla u - \frac{H}{Q^2}  D^2u \nabla u  +2  \frac{ \Delta u H  }{Q^2}   \nabla  u
     \right).
\end{align*}
Combining the computations above, we arrive at the following representation:
\begin{align*}
    \Delta_g H + \frac{1}{2}H^3-2H \Kk
    = &\ \Delta ^2 u - \Delta \left( \frac{|\nabla u|^2}{Q(1+Q)}\Delta u +\frac{\nabla u (D^2 u \nabla u)}{Q^3} \right)   -D_{ij}^2 \left( \frac{\nabla_i u \nabla _j u}{Q^2} H \right) \\
 &\ 
     +  \div  \left( \frac{5}{2}\frac{H^2 }{Q} \nabla u + 2\frac{H}{Q^2}  D^2u \nabla u  -2  \frac{ \Delta u H  }{Q^2}   \nabla  u
      \right).
\end{align*}
Using the star notation introduced in \eqref{eq:Hstar}, where $ H=  D^2 u \star \sum_{k=1}^2 Q^{-2k+1}P_{2k-2}(\nabla u)$, we can express the right-hand side in an abstract polynomial form. This yields the final reformulation of the Willmore equation:
\begin{align*}
    \Delta_g H + \frac{1}{2}H^3-2H \Kk 
    = &\ \Delta ^2 u + \Delta \Big(D^2u \star P_2 (\nabla u) \star \big(Q(1+Q)\big)^{-1} \Big) + \Delta \big( D^2u\star Q^{-3}P_2(\nabla u) \big)\\
   &\  + D^2 \left(  D^2 u \star \sum_{k=1}^2 Q^{-2k-1}P_{2k}(\nabla u)\right) \\
    &\
     +  \div  \Big(  D^2 u \star D^2 u \star \nabla u \star Q^{-1}\big(
     Q^{-2} + Q^{-4}P_{2}(\nabla u) +Q^{-6}P_{4}(\nabla u) \big) 
      \Big)\\
    &\ +\div  \left(  D^2 u \star D^2 u \star \sum_{k=1}^2 Q^{-2k-1}P_{2k-1}(\nabla u) 
      \right).
\end{align*}
Altogether, the graphical Willmore equation can now be written as
\begin{align*}
     0 = \Delta_g H + \frac{1}{2}H^3-2H \Kk= \Delta^2 u - D_{i} b^{i}_1[u] - D^2_{ij} b^{ij}_2[u],  
\end{align*}
where the divergence terms $b^{i}_1[u]$ and $b^{ij}_2[u]$ are composed of combinations of lower-order terms as expressed above.
\end{proof}

To clarify the significance of the preceding lemma, we make two remarks regarding the structure of $b_{\ell}$ for $\ell=1,2$. First, the lower-order terms appear entirely in divergence form. Notably, in contrast to the graphical results for the Willmore flow in \cite{koch2012geometric}, the non-divergence term $b_0$, which appears in formulations of the form
\[\Delta^2 u = b_0[u] + D_{i} b^{i}_1[u] + D^2_{ij} b^{ij}_2[u] , \] is absent in our setting.
Second, the term $b_1[u]$ acts structurally like $D^2u\star D^2u \star {}$(bounded terms) while  $b_2[u]$ takes the form $D^2u\star{}$(bounded terms).  Both terms involve at least one factor of $\nabla u$, contain at most two factors of $D^2u$, and are at least quadratic in terms of the first and second derivatives of $u$, but do not include higher-order derivatives. As a consequence of this structure, it becomes possible to lower the regularity assumptions from $C^{4+\alpha}$ to the spaces involving derivatives only up to second order.

\section{Preliminaries}
\label{sec:weightsob}

We denote by $W^{m,p}(\Omega)$ the Sobolev space of functions on $\Omega$ with weak derivatives up to order $m\in \N_0$ in $L^p(\Omega)$, for $p\ge 1$. The standard Lebesgue space is denoted  $L^p(\Omega)$.
 For $\ell\in\R_{\ge0}$ we say that the boundary $\partial\Omega$ is $C^\ell$-smooth, following the definition in  \cite[p. 94 Subsection 6.2]{gilbarg2001elliptic}. 
In line with \cite[Subsection 6.1]{maz2010dirichlet} we define a set $\Omega\subset\R^2$ to be a \emph{Lipschitz bounded domain}, if there exists a finite open covering $\{B_i\}_{i=1}^N$ of $\partial\Omega$ such that, for each $i \in \{1, \ldots, N\}$, after a rigid motion of $\R^n$ the intersection $B_i\cap\Omega$ coincides with the segment of $B_i$ lying in the over-graph of a Lipschitz function $\psi_i \colon \mathbb{R} \to \mathbb{R}$. Subsequently, we define the \emph{Lipschitz constant} of a bounded Lipschitz domain $\Omega \subset \mathbb{R}^2$ as:
\begin{equation}
    \text{inf}\, \text{max} \{ \|\nabla \psi_i\|_{L^\infty(\mathbb{R})} : 1 \leq i \leq N \}, 
\end{equation}
where we take the infimum over all possible finite open coverings of $\partial\Omega$ with corresponding Lipschitz representations. For instance,  a $C^1$-smooth bounded domain has a Lipschitz constant of zero, whereas a square has a Lipschitz constant equal to one. This definition is adapted from the concept of minimally smooth domains as described in \cite[Subsection 3.3 p 189]{stein1970singular}. In addition, every Lipschitz domain admits a well-defined surface measure on $\partial\Omega$ and an outward unit normal vector $\nu$ that exists almost everywhere on $\partial\Omega$ with respect to this measure, due to Rademacher's Theorem.

Furthermore, since we have reformulated the Willmore equation as an inhomogeneous biharmonic equation and the $\Delta^2$ is a higher-order elliptic operator, we can adopt the framework of variational solutions. Accordingly, we say that a function $u$ is a variational solution to \eqref{eq:Willmoreellipt} if the following identity holds:
\begin{align}
    \forall v\in C_0^\infty(\Omega): \quad 0= \int_\Omega \Delta u \Delta v \diff x +  \int_\Omega b^i_1[u]D_i v \diff x  -\int_{\Omega} b_2^{ij}[u]D ^2_{ij}  v \diff x, \label{eq:weakwillmore}
\end{align}
where we use Einstein's summation convention. For this weak formulation to be well-defined, the function $u$ must belong to the space  $W_\text{loc}^{2,2}(\Omega)$.

\subsection{Weighted Sobolev Spaces} \label{subsec:weightedsobspace}
We begin by recalling the definitions of weighted Lebesgue and Sobolev spaces. The weights we consider are given by powers of the distance function to the boundary, defined as $d(x):=\dist(x,\partial\Omega)$. In this context, we follow the description given in \cite{maz2011recent} and \cite{maz2010dirichlet} up to embedding theorems. First, for each $1 \le p\le \infty$ and $\beta\in\R$  we define the weighted Lebesgue space $L^p(\Omega; d^\beta)$ as a set of all measurable functions $u$ on $\Omega$ such that
\begin{align*}
    \|u\|_{L^p(\Omega;d^\beta)}:= \left(   
    \int_\Omega | u(x)|^p d(x)^{\beta} \diff x
    \right)^{1/p}<\infty.
\end{align*}
In particular, the parameters should lie in the follwing range
\begin{align}
    p\in(1,\infty), \qquad a\in \left(-\frac{1}{p}, 1- \frac{1}{p} \right), \label{eq:adef} 
\end{align}
although we will primarily consider the case $a\ge0$. With this, we define the weighted Sobolev space $W^{m,a}_p(\Omega)$ as the space of real-valued functions $u\in L^p_{\text{loc}}(\Omega)$ for which all weak derivatives $D^\alpha u\in L^p_{\text{loc}}(\Omega)$  exist for all $|\alpha|\le m$ and
\begin{align}
  \|u\|_{ W^{m,a}_p(\Omega)} :=\left(  \sum_{ |\alpha|\le m } 
    \|D^\alpha u \|^p_{L^p(\Omega; d^{ap})} 
    \right)^{1/p}<\infty. \label{eq:normweight1}
\end{align}
It is important to notice that when $\partial\Omega$ is a Lipschitz boundary and $a=0$, the space $W^{m,0}_p(\Omega)$ coincides with the classical unweighted Sobolev space $W^{m,p}(\Omega)$. Moreover, for $m=0$, we recover the weighted Lebesgue space: $L^p(\Omega; d^{ap}) = W^{0,a}_p(\Omega)$. It is also known that if $\Omega$ is a bounded Lipschitz domain, $C^\infty(\overline{\Omega})$ is dense in $W^{m,a}_p(\Omega)$, just as in the unweighted case. Thus, we have
\begin{align*}
    W^{m,a}_p(\Omega)= \big\{ \text{ closure of $C^\infty(\overline{\Omega})$ in } W^{m,a}_p(\Omega) \big\},
\end{align*}
with respect to the norm \eqref{eq:normweight1}. For the proofs, we refer to \cite[p.55 Theorem 7.2]{kufner1980weighted} for the case  $a\ge0$ and to \cite[p.119 Remarks 11.12 (iii)]{kufner1980weighted} for the case  $a<0$. 

Furthermore, we also need the homogeneous weighted Sobolev space, as introduced in \cite[Subsection 6.1]{maz2010dirichlet}, along with its dual space. They are defined as
\begin{align}
   \mathring{W}_p^{m,a}(\Omega):= \big\{ \text{ closure of $C^\infty_c(\Omega)$ in } W^{m,a}_p(\Omega) \big\}, 
   \quad
   W^{-m, -a}_{p'}(\Omega):= \left( \mathring{W}_p^{m,a}(\Omega) \right)^{*}, \label{eq:wwdual}
\end{align}
with \emph{dual exponent} $p'=p/(p-1)$. According to \cite[p.18 Theorem 3.6]{kufner1980weighted} the spaces $W^{m,a}_p(\Omega)$ and $\mathring{W}_p^{m,a}(\Omega),W^{-m, -a}_{p'}(\Omega)$ are both separable Banach spaces. 

\subsection{Weighted Embeddings}
\label{subsec:weightedemb}
Analogously to the unweighted case, weighted Sobolev spaces admit a rich theory of embedding theorems into spaces of continuous or Hölder continuous functions, as well as into other weighted Sobolev spaces. Due to the additional parameter, namely, the power of the weight, the weighted setting gives rise to a wider variety of embedding scenarios than the unweighted case. For bounded Lipschitz domains, the foundational results were developed by Kufner, Brown, and Opic in \cite{kufner1980weighted}, \cite{opic1990hardy}, \cite{brown1992embeddings}, and \cite{brown1998some}. In the following lemma, we summarize the embeddings that will be used in the remainder of this work.

\begin{lemma} \label{lem:cemb}
Let $\Omega\subset\R^n$ be a bounded domain with Lipschitz boundary and $p$ and $a$ satisfy the conditions given in \eqref{eq:adef}. Then the following embedding results hold:
\begin{enumerate}
\item[\cbm a] There is a continuous embedding
\begin{align*}
   \mathring{W}^{1,a}_p(\Omega)
   \hookrightarrow
   L^p(\Omega; d^{ap-p}).
\end{align*}
    \item[\cbm b] Suppose that $q\ge p$ with $\frac{n}{q}-\frac{n}{p}+1 > 0$, $a\le 0, \gamma\in \R$
    are such that
     \begin{align*}
     \frac{\gamma+n}{q} - \frac{ap+n}{p}+1> 0.
\qquad
\text{ Then it follows}
\qquad
        \mathring W^{1,a}_p(\Omega) \hookrightarrow\hookrightarrow L^q(\Omega; d^{\gamma}),
    \end{align*}
    thus, this embedding is compact.
\item[\cbm c] 
   If $p>n$, $a\le1- \frac{n}{p}$ and $0<\lambda< 1-a-\frac{n}{p}
$ then the compact embedding
\begin{align*}
    \mathring{W}^{1,a}_p(\Omega) \hookrightarrow \hookrightarrow 
    C^{\lambda}(\overline{\Omega}),
\end{align*}
  holds.
\item[\cbm d] Assume $q\ge p$ with $\frac{n}{q}-\frac{n}{p}+1 > 0$, $ \gamma\in \R$
    satisfying 
     \begin{align*}
     \frac{\gamma+n}{q} - \frac{n-1}{p}> 0.
\quad
\text{Then there is a compact embedding}
\quad
        W^{1,a}_p(\Omega) \hookrightarrow\hookrightarrow L^q(\Omega; d^{\gamma}).
    \end{align*}
\end{enumerate}
\end{lemma}
\begin{proof} To begin, we recall some notation commonly used in the literature on weighted embedding theorems. For real numbers $\alpha,\beta\in\R$  we define a weighted Sobolev space in which the function and its first derivatives are measured with different weights:
\begin{align*}
    W^{1,p}(\Omega; d^\beta, d^\alpha):= 
    \left\{ u \in L^p(\Omega;d^\beta)
    \, \left| \, \|u\|_{W^{1,p}(\Omega; d^\beta, d^\alpha)}
    = \|u\|_{L^p(\Omega;d^\beta)}+ \|Du\|_{L^p(\Omega;d^\alpha)}<\infty\right\} \right. .
\end{align*}  
The corresponding space $ W^{1,p}_0(\Omega; d^\beta, d^\alpha)$ is defined as the closure of  $C^\infty_c(\Omega)$ with respect to the norm $\|\, . \, \|_{W^{1,p}(\Omega; d^\beta, d^\alpha)}$.

\noindent \cb a  We apply \cite[p.274 Theorem 19.10]{opic1990hardy} with the parameters $\beta=ap, \alpha=ap-p$ and $q=p, \kappa=1$. This yields the continuous embedding
\begin{align*}
   \mathring{W}^{1,a}_p(\Omega)
   =
   W^{1,p}_0(\Omega;d^{\beta},d^{\beta})
   \hookrightarrow
   L^p(\Omega; d^\alpha)=L^p(\Omega; d^{ap-p}).
\end{align*}

\noindent \cb b
In the work of Opic and Kufner \cite[p.275 Theorem 19.12]{opic1990hardy} we set $\kappa =1$, since $\Omega\in C^{0,1}$  has a Lipschitz boundary, and choose $\beta = ap$ and $\alpha=\gamma$ which implies   $\beta \le 0$. 

\noindent \cb c The embeddings of the weighted Sobolev spaces into the spaces of continuous functions were established by R. C. Brown and B. Opic in \cite{brown1992embeddings}. To apply their results, we take $A=\partial \Omega$ as a singular set and $\Omega_a=\Omega$ in \cite{brown1992embeddings} on pages 282 and 283. Then the condition (A1) in Chapter 3 in \cite{brown1992embeddings} is satisfied with $r(t) =d(t)/2$ since
 \begin{align*}
     \forall t\in \Omega: \quad\overline{B}(t,d(t)/2)\not=\varnothing
     \quad \text{ and also }
     \quad
     \bigcup_{t\in\Omega}{B}(t,d(t)/2) = \Omega.
 \end{align*}

For the Hölder embedding we aim to apply the results from \cite[p. 292 (5.8)-(5.9)]{brown1992embeddings}, using the corresponding definitions therein
\begin{align*}
     C(\Omega; d^{\beta_0})
    &= \left\{ u\in C(\Omega) \, 
    \left| \, 
    \sup_{\substack{s\in\Omega}} \left( d^{\beta_1}(s)|u(s)|\right)<\infty \right\} \right.,\\
    C^{0,\lambda}(\Omega; d^{\beta_0},d^{\beta_1},\{\mathcal{U}(t)\})
    &= \left\{ u\in C(\Omega; d^{\beta_0}) \, 
    \left| \, 
    \sup_{\substack{s\in \mathcal{U}(t), s\not=t}} \left( d^{\beta_1}(s)\frac{|u(s)-u(t)|}{|s-t|^\lambda}\right)<\infty \right\} \right. .
\end{align*}
By choosing $\{\mathcal{U}(t) \}=\{B(t,\frac{1}{2}d(t))\}$, $\beta_0=0=\beta_1, \gamma=\alpha-p$ we obtain
the inequality \begin{align*}
 0=   \beta_1 p  > \gamma \left(1-\frac{n}{p}-\lambda\right) + \alpha \left(\frac{n}{p}+\lambda\right)
 = ap -p+n +\lambda p,
\end{align*}
which is satisfied for $\lambda<1-a-\frac{n}{p}$. Hence, the inequality \cite[p. 292 (5.8) ]{brown1992embeddings} is strict. Therefore, by \cite[p. 292 (5.9) and p.295 Remark 5.2]{brown1992embeddings} the embedding 
\begin{align*}
    W^{1,p}(\Omega;d^{ap-p},d^{ap})=
    W^{1,p}(\Omega; d^{\alpha-p}, d^{\alpha}) \hookrightarrow\hookrightarrow C^\lambda\left(\Omega;1,1,\left\{ B\left(t,\textstyle\frac{1}{2}d(t)\right)\right\}\right) 
\end{align*}
is compact. In particular, we have the compact embedding
\begin{align*}
    \mathring{W}^{1,a}_p(\Omega) \hookrightarrow\hookrightarrow  C^\lambda(\overline{\Omega})
\end{align*}
since $\overline{\Omega}$ is compact with Lipschitz boundary and every $u \in \mathring{W}^{1,a}_p(\Omega) $ satisties $u=0$ on $\partial\Omega$.

\noindent \cb d
Again, since $\Omega\in C^{0,1}$ and there has a Lipschitz boundary, we are able to use \cite[p.275 Theorem 19.11 (19.39)]{opic1990hardy} where we choose $\kappa =1$, and put $\beta = ap$ and $\alpha=\gamma$. 
\end{proof}


From these embedding Theorems, we can draw several conclusions. First,  by applying Lemma~\ref{lem:cemb} \cbm d for $q=p$ obtain a simpler compact embedding:
\begin{align}
    \forall \gamma>-1: \qquad W^{1,a}_p(\Omega)  \hookrightarrow\hookrightarrow L^p(\Omega; d^{\gamma}). \label{eq:embpq}
\end{align}
In particular, this yields the compact embedding $W^{1,a}_p(\Omega)  \hookrightarrow\hookrightarrow L^p(\Omega)$. Moreover, by considering the constant function $u\equiv 1$ and choosing $a=0$, we observe that the distance to boundary function  $d$  satisfies $d^\gamma \in  L^1(\Omega)$ for all $\gamma>-1$, provided that $\Omega$ has a Lipschitz boundary.

\subsection{Dirichlet Data Spaces}
\label{subsec:dirichdt}
Regarding the trace theory of the weighted Sobolev spaces on Lipschitz domains, we rely on the comprehensive results presented in \cite[p. 208 Chapter 7]{maz2010dirichlet}. Specifically, we define the appropriate trace spaces for Dirichlet data of the form $\big( u, \frac{\partial u}{ \partial \nu} \big)$ where the normal derivative is given by $\frac{\partial u}{ \partial \nu} = \langle \nu ,\operatorname{Tr}[\nabla u]\rangle$ on $\partial\Omega$ in the context of fourth-order boundary value problems. 

To this end, we begin by recalling the definition of the Besov space on the boundary $B^s_p(\partial\Omega)$ where $p\in(1,\infty)$ and the smoothness parameter $s\in(0,1)$ depends on $a$ in the following way
\begin{align}
    s:=1-a - \frac{1}{p} \in(0,1). \label{eq:sdef} 
\end{align}
 The space  $B^s_p(\partial\Omega)$ consist of functions   $f\in L^p(\partial\Omega)$ such that
\begin{align*}
 \| f\| _{B^s_p(\partial\Omega)}:= \| f\|_{L^p(\partial\Omega)} 
 + \left(\int_{\partial\Omega}\int_{\partial\Omega} \frac{\big| f(x)-f(y) \big|^p}{|x-y|^{n-1+sp}} \diff S_x \diff S_y \right)^\frac{1}{p}< \infty,
\end{align*}
where $\diff S$ denotes the line element on $\partial\Omega$. 
Indeed, as shown in \cite[Lemma 1, p. 39]{maz2011recent}, the space $B^s_p(\partial\Omega)$ serves as the trace space of $W^{1,a}(\Omega)$  in the following sense: For parameters $a$ and $s$ satisfying \eqref{eq:adef} and \eqref{eq:sdef}, respectively, the trace operator
$
\operatorname{Tr}: W^{1,a}_p(\Omega) \to B^s_p(\partial\Omega)
$
is well-defined,  bounded, onto,  linear, and has the homogeneous space $\mathring{W}^{1,a}_p(\Omega)$ as its null space. Moreover, there exists an extension  operator $\operatorname{E}:  B^s_p(\partial\Omega) \to W^{1,a}_p(\Omega) $, which is also linear and continuous and satisfies $\operatorname{Tr}\circ \operatorname{E}= \operatorname{Id}$. 

Next, we turn to the space $\dot W_p^{1+s}(\partial\Omega)$ which is relevant for prescribing Dirichlet boundary data in the biharmonic case $m=2$. 
According to \cite[Corollary 7.11]{maz2010dirichlet}  this space admits a rather simple characterization. To describe it, we first introduce the tangential derivative, defined by \cite[Subsection 7.2]{maz2010dirichlet} as
\begin{align}
    \nabla_\text{tan}=  \nabla - \left\langle \nu, \frac{\partial}{\partial\nu}\right\rangle. \label{eq:tanggrad}
\end{align}
This definition allows us to formulate Sobolev spaces of order one on the boundary  $\partial \Omega$, which play a key role in the case $m=2$ and will be used later on. Let $\varphi\colon I\to \partial\Omega$ parametrization by arclength, then for $1<p<\infty$ we define 
\begin{align*}
    L^1_p(\partial\Omega):= \big\{ f\circ \varphi \in W^{1,p}(I) \, \big| \, \|f\|_{L^1_p(\partial\Omega)}:=\| f\|_{L^p(\partial\Omega)}+ \| \nabla_\text{tan} f \|_{L^p(\partial\Omega)}<\infty \big\}.
\end{align*}
For $p\in(1,\infty)$ and $s\in (0,1)$,   we define the space $\dot{W}^{1+s}_p(\partial\Omega)$ by \cite[p. 232 Corollary 7.11]{maz2010dirichlet} as
\begin{align}
    \dot{W}^{1+s}_p(\partial\Omega) := \big\{ (g_0,g_1) \in L^1_p(\partial \Omega) \oplus  L^p(\partial\Omega)\ \big| \ \nu g_1 + \nabla_\text{tan} g_0 \in B^s_p(\partial\Omega) \big\} \label{eq:dirichspace}
\end{align} 
equipped with the norm
\begin{align*}
    \|g\|_{\dot W^{1+s}_p(\partial\Omega)}:= \|g_0\|_{B^s_p(\partial\Omega)}
    +\|\nu g_1 + \nabla_{\rm tan} g_0\|_{B^s_p(\partial\Omega)}.
\end{align*}

Trace and extension results for this space are established in \cite[p. 223, Theorem 7.8]{maz2010dirichlet}. In particular, there exists a well-defined and bounded trace operator $  \operatorname{Tr}_{1}: W^{2,a}_p(\Omega) \ni u \mapsto \big\{ \frac{\partial^k u}{\partial \nu^k} \big\} _  {0\le k\le 1 }\in\dot W^{1+s}_p(\partial\Omega) $ which admits a bounded linear right inverse $  \operatorname{Ext}_{1}: \dot W^{1+s}_p(\partial\Omega)    \to W^{2,a}_p(\Omega) $ such that $\operatorname{Tr}_{1}\circ\operatorname{Ext}_{1}= {\rm Id}$. The kernel of the trace operator is precisely the homogeneous space $\mathring W^{2,a}_p(\Omega)$ consisting of functions that vanish along with their normal derivative on the boundary:
\begin{align}
  \mathring W^{2,a}_p(\Omega)=\left\{  u\in W^{2,a}_p(\Omega)\, \left| \, u=0,\ \frac{\partial u}{\partial \nu} = 0 \text{ on } \partial\Omega \right\} \right. . \label{eq:rindwbound}
\end{align}

\subsection{Weighted Sobolev Estimates and Existence Results} \label{eq:weightedsobolev}
Having now introduced weighted Sobolev spaces and the corresponding Dirichlet trace spaces, we revisit the inhomogeneous biharmonic Dirichlet problem:
\begin{align}
    \begin{cases}
 \Delta^2 u= \mathcal{F} \quad \text{ for } x \in \Omega, \\
    \displaystyle\frac{\partial^k u}{\partial \nu^k} = g_k \quad \text{ on } \partial \Omega, \quad 0\le k\le 1.
    \end{cases} \label{eq:weightedlp} 
\end{align}
where the right-hand side satisfies $\mathcal {F} \in W^{-2, a}_p(\Omega),$ and the Dirichlet boundary data are given by $ g:=\{g_k\}_{0\le k\le 1}\in \dot W^{1+s}_p(\partial \Omega)$. 
Solvability and uniqueness results for the problem \eqref{eq:weightedlp} in the weighted Sobolev–Besov setting were established in \cite[p. 169 Theorem 1.1]{maz2010dirichlet}, under additional assumptions on the bounded Lipschitz domain  $\Omega$. In this paper, the corresponding result is stated in Proposition~\ref{thm:prop1}.  

To formulate one of the boundary assumptions required for these results, we introduce the BMO modulo VMO character of a function $f\in L^1(\partial\Omega)$ by the quantity
\begin{align*}
    \{ f\}_{\ast,\partial\Omega}:=\lim_{\varepsilon\to 0}\left(\sup_{t\in\partial\Omega} \meanint_{B_\varepsilon^n(t) \cap \partial\Omega} 
    \meanint_{B_\varepsilon^n(t) \cap \partial\Omega} \big| f(x)-f(y) \big| \diff S_x \diff S_y \right),
\end{align*}
where $B_\varepsilon^n(t)$ denotes the open ball  $n$-dimensional open ball with the center $t$ and radius $\varepsilon$. 

We are now in a position to recall the existence and regularity result for the biharmonic Dirichlet problem in the weighted Sobolev setting, as established in \cite{maz2010dirichlet}.

\begin{Proposition} \label{thm:prop1}
Let $\Omega\subset \R^n$ be a bounded Lipschitz domain with the exterior normal vector $\nu$. Further, suppose $p\in (1,\infty)$ and $s\in(0,1)$ with $a:=1-s-1/p$ according to \eqref{eq:adef} as well as $\mathcal {F} \in W^{-2, a}_p(\Omega),$ and Dirichlet boundary data $ g:=\{g_k\}_{0\le k\le 1}\in \dot W^{1+s}_p(\partial \Omega)$. Then there exists a constant $\Cl{vmobound}>0$ depending only on the Lipschitz constant of $\Omega$ such that if
\begin{align}
\label{eq:condnuA}
\{\nu\}_{\ast,\partial\Omega}
\le \Cr{vmobound}\,s(1-s)\Big(p^2(p-1)^{-1}+s^{-1}(1-s)^{-1}\Big)^{-1}, 
\end{align}
then there exists a unique solution $u\in W^{2,a}_p(\Omega)$ to the inhomogeneous biharmonic Dirichlet problem \eqref{eq:weightedlp}. Moreover, there exists a constant $\Cl{lpestimate}=\Cr{lpestimate}(\partial\Omega, p,s)$ such that
\begin{align*}
    \|u\|_{W^{2,a}_p(\Omega)}\le \Cr{lpestimate} \left( \|g\|_{\dot W^{1+s}_p(\partial\Omega)} + \|\mathcal F \|_{W^{-2,a}_p(\Omega)} \right).
\end{align*}
\end{Proposition}
\begin{proof}
The result follows directly from \cite[ Theorem 8.1, p.233]{maz2010dirichlet}. In our case, the biharmonic operator has constant elliptic coefficients  $A_{\alpha\beta}$, which satisfy the assumptions of the theorem.
\end{proof}

The condition \eqref{eq:condnuA} is  satisfied for all values $p\in(1,\infty)$, $s\in(0,1)$, when $\partial \Omega\in C^1$. The same condition also holds if the BMO modulo VMO character $\{\nu\}_\ast$ is sufficiently small. With regard to boundary regularity, it is worth noting that, depending on $p$ and $s$, the smallness of $\{\nu\}_\ast$ can be ensured by imposing a sufficiently small Lipschitz constant on the boundary. As observed in \cite[p.43]{maz2011recent}, one class of such domains includes Lipschitz graph polyhedral domains whose dihedral angles are chosen sufficiently close to $\pi$, with the closeness depending on $p$ and $s$. This means that the Lipschitz boundary exhibits only mild angular deviations (small kinks).

\begin{remark} \label{remark:lipsch}
   According to Theorem 2 on page 44 in \cite{maz2011recent}, the condition \eqref{eq:condnuA} can be omitted if one restricts the range of the parameters $p$ and $a$. More precisely, there exists a constant $\varepsilon>0$ depending only on the upper bound of the Lipschitz constant of $\partial\Omega$, such that if
\begin{align}
\label{eq:condpae}
\left( \frac{1}{2} - \frac{1}{p} \right) < \varepsilon, \quad |a|<\varepsilon, 
\end{align}
then the same existence results hold. As $\varepsilon\to 0$, we obtain $s\to 1/2$, indicating a trade-off between the regularity of the domain boundary and the regularity required of the Dirichlet data.
\end{remark}

\section{Existence and Regularity Results for the Willmore Equation}
\label{sec:existencereg}

In this section, we first prove the more abstract Theorem~\ref{thm:lipschitz}. To this end, we establish the existence of a solution to the Willmore boundary value problem in the framework of weighted Sobolev spaces in Subsection~\ref{subsec:mainweighted}. The Hölder case, stated as Theorem~\ref{thm:weightedmainhoelder} in the introduction, will then be derived as a corollary in Subsection~\ref{subsec:hoelderboundary}.

\subsection{Preliminary Results} \label{subsec:prelim}
We begin with some preliminary estimates. The first result concerns general biharmonic problems with a right-hand side in divergence form. Moreover, to apply the weighted embedding $\mathring W^{2,a}_p(\Omega)\hookrightarrow\hookrightarrow C^1(\overline{\Omega})$ from Lemma~\ref{lem:cemb} (c), we must restrict the weight exponent to satisfy $a< 1-\frac{2}{p}$.

\begin{Proposition}
 \label{weightschauderlpdiv}
Let $\Omega\subset \R^2$ be a bounded Lipschitz domain with exterior normal vector $\nu$ satisfying \eqref{eq:condnuA}. Furthermore, assume that 
$p>2$, $t\in \big(\frac{2p}{2+p-ap},p\big]$, $a\in \big(0, 1-\frac{2}{p}\big)$ and 
$    h_1\in L^{t}(\Omega ;d^{2at})$, $
  h_2 \in L^{p}(\Omega ;d^{ap})$ as well as $g:=\{g_0,g_1\}\in \dot W^{1+s}_p(\partial\Omega)$ such that $g_0, \nu g_1+\nabla_{\rm tan} g_0\in L^{\infty}(\partial\Omega)$ with $s:=1-a-\frac{1}{p}$. Then the following Dirichlet problem
\begin{align}
    \begin{cases}
   \Delta^2u =  D_ih^i_1+D^2_{ij}h^{ij}_2 \quad \text{ in } \ \Omega, \\
    \displaystyle\frac{\partial^k u}{\partial \nu^k} = g_k \quad \text{ on } \ \partial \Omega, \quad 0\le k\le 1,
    \end{cases} \label{eq:lpexistest} 
\end{align}
admits a unique variational solution $u\in W^{2,a}_p(\Omega)\cap C^1(\overline{\Omega})$. Moreover, there exist constants $\Cl{w3202}=\Cr{w3202}(p,t,a,\partial\Omega),$ $\Cl{w32021}=\Cr{w32021}(a,\partial\Omega)$ such that
\begin{align*}
    \|u\|_{W^{2,a}_p(\Omega)} &\le \Cr{lpestimate}
   \|  
    g  \|_{ \dot W^{1+s}_p (\partial\Omega)} 
     + 
        \Cr{w3202} \left(
    \|h_1\|_{L^{t}(\Omega ;d^{2at})}
   + \|h_2\|_{L^{p}(\Omega ;d^{ap})}
    \right)
    ,\\
         \|D^2u\|_{L^{1}(\Omega;d^a  )} &\le \Cr{w32021}
   \| \nu g_1
   +
   \nabla_{\rm tan} g_0\|_{L^\infty(\partial\Omega)}
     + 
        \Cr{w3202} \left(
    \|h_1\|_{L^{t}(\Omega ;d^{2at})}
   + \|h_2\|_{L^{p}(\Omega ;d^{ap})}
    \right),
    \\
            \|\nabla u\|_{L^\infty(\Omega)} &\le \Cr{miranda1}
   \| \nu g_1
   +
   \nabla_{\rm tan} g_0\|_{L^\infty(\partial\Omega)}
     + 
        \Cr{w3202} \left(
    \|h_1\|_{L^{t}(\Omega ;d^{2at})}
   + \|h_2\|_{L^{p}(\Omega ;d^{ap})}
    \right)
    .
\end{align*}
\end{Proposition}
\begin{proof} We now aim to apply the embedding Lemma~\ref{lem:cemb} and Proposition~\ref{thm:prop1}, specialized to the case $n = 2$.
To derive the $C^1(\overline{\Omega})$ estimate, we split the system \eqref{eq:lpexistest} into two parts:
\begin{align*}
\begin{aligned}
            \begin{cases}
   \Delta^2w = 0  \quad \text{ in } \ \Omega, \\
    \displaystyle\frac{\partial^k w}{\partial \nu^k} = g_k \quad \text{ on } \partial \Omega, \quad 0\le k\le 1,
    \end{cases} 
\end{aligned}
\qquad 
    \begin{cases}
   \Delta^2v =  D_ih^i_1+D^2_{ij}h^{ij}_2 \quad \text{ in }  \ \Omega, \\
    \displaystyle\frac{\partial^k v}{\partial \nu^k} = 0 \quad \text{ on } \partial \Omega, \quad 0\le k\le 1.
    \end{cases} 
\end{align*}

\noindent\cb{1}
We begin with the function $v$.  In this context, by the definition in \eqref{eq:wwdual}, we need to verify that $D_ih^i_1+D^2_{ij}h^{ij}_2\in W^{-2,a}_p(\Omega)=\big(\mathring W^{2,-a}_{p'}(\Omega)\big)^\ast$. Therefore, let $\varphi\in C^\infty_c(\Omega)$ and let $t'=t/(t-1)$ denote the dual exponent of $t$. Then we compute:
\begin{align*}
 \left|\int_\Omega \varphi D h_1 \diff x\right|=  \left| \int_\Omega D \varphi h_1 \diff x \right|
 \le
 \int_\Omega |D \varphi| d^{- 2a } d^{2a} |h_1| \diff x 
    \le \|D \varphi  \|_{L^{t'}(\Omega; d^{-2at' })}  \| h_1  \|_{L^t(\Omega; d^{2at}) }
\end{align*}
by Hölder's inequality. Thus, it remains to establish the following estimate:
\begin{align}
    \|D \varphi  \|_{L^{t'}(\Omega; d^{-2at' })} \le  \C \|D \varphi  \|_{\mathring W^{1,-a}_{p'}(\Omega)}. \label{eq:hilfe}
\end{align}
Hence, it remains to verify that the assumptions of Lemma~\ref{lem:cemb} \cbm b are satisfied. First, we set $t'\to q$ and $p'\to p$ in Lemma~\ref{lem:cemb} \cbm b Note that this substitution requires $q\ge p$, i.e. $t' \ge p'$, which is indeed fulfilled.
Furthermore, we check the condition
\begin{align*}
    \frac{2}{t'}-\frac{2}{p'}+1 = \frac{2}{p}- \frac{2}{t} +1 > \frac{2}{p}-\frac{p+2}{p} +1 =0,
\end{align*}
which holds because $t>\frac{2p}{2+p-ap}\ge \frac{2p}{2+p} $. 
The second condition, we additionally set $\gamma\to -2at'$   in Lemma~\ref{lem:cemb} \cbm b and obtain
\begin{align*}
    \frac{-2at'}{t'}- \frac{-ap'}{p'}+ \frac{2}{p}- \frac{2}{t} +1 
    &\ = -a +\frac{2}{p}- \frac{2}{t} +1 > 0.
\end{align*}
Therefore, the compact embedding ${\mathring W^{1,-a}_{p'}(\Omega)}
 \hookrightarrow\hookrightarrow{L^{t'}(\Omega; d^{-2at' })}$ holds and the estimate \eqref{eq:hilfe} follows.  
Next, we consider the contribution from the $h_2$-term.  For any test function $\varphi\in C^\infty_c(\Omega)$ we compute:
\begin{align*}
 \left|\int_\Omega \varphi D^2 h_2 \diff x\right|=  \left| \int_\Omega D^2 \varphi h_2  \diff x \right|
 \le
 \int_\Omega |D^2 \varphi| d^{- a} d^{a} |h_2| \diff x 
    \le \|D^2 \varphi  \|_{L^{p'}(\Omega; d^{-ap' })}  \| h_2  \|_{L^p(\Omega; d^{a p}) }.
\end{align*}
Combining this with the estimate for the $h_1$-term, we obtain
\begin{align}
    \left| \int_\Omega \varphi(D_ih^i_1+D^2_{ij}h^{ij}_2) \diff x \right| \le \|\varphi\| _{\mathring W^{2,-a}_{p'}(\Omega)} \cdot \C \left(
    \|h_1\|_{L^{t}(\Omega ;d^{2at})}
   + \|h_2\|_{L^{p}(\Omega ;d^{ap})}
    \right). \label{eq:hilfe1}
\end{align}
Since $C^\infty_c(\Omega)$ is dense in $\mathring W^{2,-a}_{p'}(\Omega)$ the linear functional
\begin{align*}
    \mathring W^{2,-a}_{p'} (\Omega) \ni u \mapsto \int_{\Omega} u( D_ih^i_1+D^2_{ij}h^{ij}_2) \diff x 
\end{align*}
defines a bounded element in $W^{-2,a}_{p}(\Omega)$ with its norm bounded by a multiple of   $ \left(     \|h_1\|_{L^{t}(\Omega ;d^{2at})}    + \|h_2\|_{L^{p}(\Omega ;d^{ap})}     \right) $. Applying Proposition~\ref{thm:prop1}, we conclude existence and uniqueness of a solution $v\in W^{2,a}_p(\Omega)$  to the homogeneous Dirichlet problem as well as the corresponding $W^{2,a}_p(\Omega)$-a-priori estimate. 
 Moreover, since $\frac{\partial^k v}{\partial \nu^k} = 0$ on $ \partial \Omega$ for $ 0\le k\le 1$ it follows from \eqref{eq:rindwbound} that $v\in \mathring W^{2,a}_p(\Omega)$ and we obtain the following estimate
\begin{align*}
    \|v\|_{\mathring W^{2,a}_p(\Omega)} \le \Cr{lpestimate}\C \left(
    \|h_1\|_{L^{t}(\Omega ;d^{2at})}
   + \|h_2\|_{L^{p}(\Omega ;d^{ap})}
    \right).
\end{align*}
Applying Hölder's inequality and using the boundedness of $\Omega$, we further deduce
\begin{align*}
    \|D^2v\|_{L^1(\Omega;d^a)} \le \C \left(
    \|h_1\|_{L^{t}(\Omega ;d^{2at})}
   + \|h_2\|_{L^{p}(\Omega ;d^{ap})}
    \right).
\end{align*}
Furthermore, since $p>2$ and $ a<1-\frac{2}{p}$, the weighted embedding $\mathring W^{2,a}_p(\Omega)\hookrightarrow\hookrightarrow C^1(\overline{\Omega})$ holds by Lemma~\ref{lem:cemb} \cbm c. Hence, we conclude
\begin{align*}
    \|v\|_{C^1(\overline{\Omega})} \le \C \left(
    \|h_1\|_{L^{t}(\Omega ;d^{2at})}
   + \|h_2\|_{L^{p}(\Omega ;d^{ap})}
    \right).
\end{align*}

\noindent\cb{2}
We now turn to the biharmonic Dirichlet problem associated with the function $w$, which corresponds to the inhomogeneous boundary data. Once again, we apply Proposition~\ref{thm:prop1} to obtain existence and uniqueness of a solution $w\in W^{2,a}_p(\Omega)$, along with the estimate
\begin{align*}
    \| w\|_{W^{2,a}_p(\Omega)}\le \Cr{lpestimate} \|g\|_{\dot W^{1+s}_p(\partial\Omega)}.
\end{align*}
The following $L^\infty$-gradient estimate relies on the Agmon-Miranda maximum modulus estimate, as proved by Pipher and Verchota for Lipschitz domains.  In order to apply their result, we need to relate the Dirichlet boundary data $g\in \dot W^{1+s}_p(\partial\Omega)$ to the space $WA^{1,\infty}(\partial\Omega)$ (for notation see \cite{Pipher1993}). At this point, we observe, that by \cite[p.230 Corollary 7.10]{maz2010dirichlet} there exists an array in non-Dirichlet trace space $f\in \dot B^{1+s}_p(\partial\Omega)$ such that 
\begin{align*}
    g_0= f_0, \quad g_1= \nu f_1 \quad \text{ and } \quad f_1=\nu g_1+ \nabla_{\rm tan} g_0.
\end{align*}
By the trace theorems, it follows that ${\rm Tr}(u)=g_0, {\rm Tr}(\nabla u) =f_1= \nu g_1+ \nabla_{\rm tan} g_0$. Now since $g_0\in L_2^1(\partial\Omega),  g_1\in L_2(\partial\Omega)$ we may invoke \cite[p. 387 Theorem 1.2]{Pipher1993} to obtain the estimate
\begin{align*}
    \|\nabla w\|_{L^\infty(\Omega)} \le \Cr{miranda1} \|\nu g_1+\nabla_{\rm tan} g_0\|_{L^\infty(\partial\Omega)}. 
\end{align*}

\noindent\cb{3}
In the last step, we use the result concerning the domination of square functions by the non-tangential maximal function in $L^2$, as established in \cite[p.1455 Theorem 2]{dahlberg1997area} and obtain
\begin{align*}
    \int_\Omega |D^2w(x)|^2 d(x) \diff x \le  \C \int_{\partial\Omega}| N(\nabla w)|^2 \diff s  \le \C \|\nabla w\| _{L^\infty(\partial\Omega)}^2,
\end{align*}
where $N(\nabla u) $ denotes the non-tangential maximal function of $\nabla u$ on boundary. Consequently, we deduce the weighted estimate
\begin{align*}
  \|D^2w\|_{L^1(\Omega;d^a)} = &\  \left|\int_\Omega |D^2 w(x)| 
    d^{\frac{1}{2}}(x)  d^{a-\frac{1}{2}}(x)
    \diff x   \right| \\
    \le &\   \C  \sqrt{\int_\Omega |D^2 w(x)| ^2
    d(x)  
    \diff x   }  \cdot \sqrt{\int _\Omega d^{2a-1}\diff x}\\
    \le &\  \C \|\nabla w\| _{L^\infty(\partial\Omega)} \cdot \sqrt{\int _\Omega d^{2a-1}\diff x}<\infty,
\end{align*}
since  $a>0$ implies $2a-1>-1$, ensuring the integrability of the weight. Combining the above results completes the proof.
\end{proof}

In the following lemma, we take preparatory steps toward the application of a fixed point argument. Specifically, we derive estimates for the terms appearing on the right-hand side in divergence form, as well as for their differences, all within the weighted function space framework.

\begin{lemma}\label{lemm:wprem12}
 Let $p\in (2,\infty), a\in \big(-\frac 1p, 1-\frac{2}{p}\big)$  and $ i,j,\ell, k\in\N_0\setminus\{0\}$ and $u\in W^{2,a}_p({\Omega})$ such that $\|\nabla u\|_{L^\infty({\Omega})}\le1$.
 
 Then there exists a  constant $\Cl{whoelder1a}=\Cr{whoelder1a}(\Omega, p,a)$ such that 
 \begin{align}
      \begin{aligned}
    \big\|D^2 u \star D^2  u \star  Q^{-j}P_{i}(\nabla u)  \big\|_{L^{p/2}(\Omega;d^{ap})} 
   \ &\
    \le \Cr{whoelder1a}\| D^2 u \|^2_{L^{p}({\Omega};d^{ap})} \|\nabla u\|_{L^\infty({\Omega})},
    \\
    \big\|D^2 u \star  Q^{-j}(1+Q)^{-\ell}P_{k}(\nabla u)  \big\|_{L^{p}({\Omega};d^{ap})} 
    \ &\  \le \Cr{whoelder1a}  
         \|D^2 u \|_{L^{p}({\Omega};d^{ap})} \|\nabla u  \|_{L ^\infty({\Omega})} 
.
\end{aligned} \label{eq:westvorlp1a} 
 \end{align}
Additionally, suppose $w\in W^{2,a}_p({\Omega})$ such that $\|\nabla w\|_{L^\infty({\Omega})}\le 1$ then it follows
\begin{align}
\begin{aligned}
    \Big\| D^2 u \star &\ D^2  u \star  Q^{-j}(u)P_{i}(\nabla u)  
\    -  D^2 w \star D^2 w  \star  Q^{-j}(w)P_{i}(\nabla w)  \Big\|_{L^{p/2}(\Omega;d^{ap})} 
    \\
    \le &\ \Cr{whoelder1a}   \| u-w \|_{W^{2,a}_p(\Omega)} \Big(  \| D^2u \|_{L^p(\Omega;d^{ap})}^2+ \| D^2w \|_{L^{p}(\Omega;d^{ap})} ^2\Big) 
\end{aligned} \label{eq:westvorlp2a}
\end{align}
as well as in case $k \ge 2$
\begin{align}
      \begin{aligned}
    \big\|D^2 u \ \star    &\ Q^{-j}(1+Q)^{-\ell}(u)P_{k}(\nabla u)
     -
    D^2 w \star  Q^{-j}(1+Q)^{-\ell}(w)P_{k}(\nabla w)  \big\|_{L^{p}({\Omega};d^{ap})} 
    \\
        \le &\ \Cr{whoelder1a}   \|u-w \|_{W^{2,a}_p({\Omega})} 
       \left(
 \|\nabla u\| _{L^\infty({\Omega})} +\|\nabla w\| _{L^\infty({\Omega})} 
               \right)   
               \left(
 \| u\| _{W^{2,a}_p(\Omega)} +\|w\| _{W^{2,a}_p(\Omega)} 
               \right)
.
\end{aligned} \label{eq:westvor21a}
 \end{align}
\end{lemma}
\begin{proof}
\noindent\cb 1 We begin by analyzing the estimate \eqref{eq:westvorlp1a}.
The inequality $Q=\sqrt{1+|\nabla u|^2}\ge 1$ and $\|\nabla u\|_{L^\infty({\Omega})}\le1$ yield
 \begin{align*}
 \begin{aligned}
    \big|D^2 u \star D^2  u \star  Q^{-j}P_{i}(\nabla u)  \big| &\le
    \C | D^2 u |^2 \cdot \left|{|\nabla u|^{i}}{\sqrt{1+|\nabla u|^2}^{(-j)}}\right|
   \\
    &
    \le  \C | D^2 u |^2 \cdot \|\nabla u\|_{L^\infty(\Omega)},\\
 \big|D^2 u \star  Q^{-j}(1+Q)^{-\ell}P_{k}(\nabla u)   \big|
     &\le \C | D^2 u | \cdot \left|{|\nabla u|^{k}}{\sqrt{1+|\nabla u|^2}^{(-j)}}\big(1+ {\sqrt{1+|\nabla u|^2}^{(-j)}} \big)^{-\ell}\right|\\
   & \le \C
         |D^2 u | \cdot \|\nabla u  \|_{L^\infty({\Omega})} .
\end{aligned} 
    \end{align*}
Applying $L^{p/2}(\Omega;d^{ap})$-norm on both sides of the first inequality and respectively $L^{p}(\Omega;d^{ap})$-norm for the second inequality, we obtain the estimates in \eqref{eq:westvorlp1a}.  

\noindent\cb 2 
In the next step,  we rewrite the following difference
\begin{align*}
    D^2 u \star D^2 & u \star  Q^{-j}(u)P_{i}(\nabla u)  
    -\ D^2 w \star D^2 w \star  Q^{-j}(w)P_{i}(\nabla w) 
    \\ = &\ (D^2 u -D^2 w)\star D^2  u \star  Q^{-j}(u)P_{i}(\nabla u)  
    +\ D^2 w \star (D^2 u - D^2 w) \star  Q^{-j}(u)P_{i}(\nabla u) \\
    &\ +\ D^2 w\star D^2 w \star \big( Q^{-j}(u)- Q^{-j}(w) \big)P_{i}(\nabla u)  
   \\
   &\ \quad  +\ D^2 w \star D^2 w \star Q^{-j}(w) \big(P_{i}(\nabla u)-P_i(\nabla w)\big). 
\end{align*}
Using again that $\| Q^{-j}\|_{L^\infty({\Omega})}  \le 1 $ as well as $\|\nabla u\|_{L^\infty({\Omega})},\|\nabla w\|_{L^\infty({\Omega})}\le 1$, we can show \eqref{eq:westvorlp2a}
\begin{align}
\begin{aligned}
    \big| D^2 u \star D^2 & u \star  Q^{-j}(u)P_{i}(\nabla u)  
    -\ D^2 w \star D^2 w \star  Q^{-j}(w)P_{i}(\nabla w) \big| 
    \\ \le &\    \big| (D^2u-D^2w) \star D^2 u\star \nabla u\big|
    +  \big|D^2 w \star (D^2 u-D^2w )\star\nabla u \big|\\
    & + j \big| D^2 w \star D^2 w\star (\nabla u- \nabla w) \big|
     + i \big| D^2 w \star D^2 w\star (\nabla u- \nabla w) \big|.
\end{aligned}
\end{align}
Here, we have used that the differences $Q^{-j}(w)-Q^{-j}(u)$ and $P_{i}(\nabla u)-P_i(\nabla w)$ can be estimated linearly in terms of  $|\nabla w - \nabla u|$. Indeed, it holds that:
\begin{align*}
   \big| Q^{-j}(u)-Q^{-j}(w) \big| \le & \ \big| \nabla u - \nabla w \big| \left|\frac{\big(\nabla u +\nabla w\big)}{\Big(Q(u) + Q(w)\Big)}\sum_{\ell=1}^{j}\frac{Q^{\ell-1}(u)Q^{j-\ell}(w)}{Q^j(u) Q^j(w)}\right|
   \le j\big| \nabla u - \nabla w \big|,\\
    \big| P_i(u)-P_i(u) \big| \le & \ \big| \nabla u - \nabla w \big| \left|\sum_{\ell=1}^{i}P_{\ell-1}(u)P_{i-\ell}(w)\right|
   \le i\big| \nabla u - \nabla w \big|.
\end{align*}
We conclude by applying the $L^{p/2}(\Omega;d^{ap})$-norm on both side and obtain
\begin{align*}
    \Big\| D^2 u \star D^2 & u \star  Q^{-j}(u)P_{i}(\nabla u)  
    -\ D^2 w \star D^2 w \star  Q^{-j}(w)P_{i}(\nabla w) \Big\|_{L^{p/2}(\Omega;d^{ap})} 
    \\
    \le &\ \C  \left(
        \| \nabla u D^2u \|_{L^{p}(\Omega;d^{ap})} +  \| \nabla u D^2w \|_{L^{p}(\Omega;d^{ap})}
\right)\|D^2 w - D^2 u\| _{L^{p}(\Omega;d^{ap})} 
 \\
  &  +  \C  \big( \| D^2w \|_{L^{p}(\Omega;d^{ap})}^2\big)\|\nabla w-\nabla u\| _{L^\infty({\Omega})}. 
\end{align*}
By Hölder's inequality $\|D^2 w  D^2 u\| _{L^{p/2}(\Omega;d^{ap})}\le \|D^2 u \| _{L^{p}(\Omega;d^{ap})}\|D^2 w\| _{L^{p}(\Omega;d^{ap})}$, Young's inequality and Sobolev embedding  $(u-w)\in \mathring W^{2,a}_p(\Omega)\hookrightarrow\hookrightarrow C^1(\overline{\Omega})$, we finish the proof of \eqref{eq:westvorlp2a}. In particular, we have:
\begin{align*}
    \|\nabla w-\nabla u\| _{L^\infty({\Omega})}& \le \C \| w- u\| _{W^{2,a}_p(\Omega)}, 
    \\ 
    \| \nabla u D^2u \|_{L^{p}(\Omega;d^{ap})} & \le \C \| u\| _{W^{2,a}_p(\Omega)}\cdot \| D^2u \|_{L^{p}(\Omega;d^{ap})}.
\end{align*}

\noindent\cb 3 Let dedicate ourselves to the proof of the last estimate \eqref{eq:westvor21a}. The argument follows a similar structure as in the previous step. Here we use $k\ge 2$ and apply the weighted $L^p(\Omega;d^{ap})$-norm on both sides of the inequality: 
\begin{align*}
\begin{aligned}
    \big| D^2 u\  \star   Q^{-j}(1+Q)^{-\ell}(u)P_{k}(\nabla u)
     & - D^2 w \star Q^{-j}(1+Q)^{-\ell}(w)P_{k}(\nabla w) \big| 
    \\ \le &\   
      \big|(D^2u-D^2w)\star \nabla u \star \nabla u \big| + j  \big|D^2 w \star (\nabla u-\nabla w)\star \nabla u \big| \\
    & +  \ell \big|D^2 w \star (\nabla u-\nabla w)\star \nabla u \big|
    +   \big|D^2 w  \star(\nabla u-\nabla w) \star \nabla u\big| \\
    &\ + (k-1)\big|D^2 w   \star \nabla w\star(\nabla u-\nabla w)\big|.
\end{aligned} 
\end{align*}
From this, we directly deduce:
\begin{align*}
        \big\| D^2 u \star   Q^{-j}& (1+Q)^{-\ell}(u)  P_{k}  (\nabla u)
    -  D^2 w \star Q^{-j}(1+Q)^{-\ell}(w)P_{k}(\nabla w) \big\|_{L^p(\Omega;d^{ap})} 
    \\ \le &\ \C
      \|\nabla u \|_{L^\infty({\Omega})}^2
     \cdot
      \|D^2u-D^2w\|_{L^{p}(\Omega;d^{ap})}  \\
       &\ + \C
      \left(\|\nabla u \|_{L^\infty(\Omega)}
      +\|\nabla w \|_{L^\infty(\Omega)}\right)
            \cdot
            \|D^2 w \|_{L^{p}(\Omega;d^{ap})}
            \cdot
      \|\nabla u- \nabla w\|_{L^\infty(\Omega)}
.
\end{align*}
The proof is completed by applying the compact Sobolev embedding
${W^{2,a}_p(\Omega)}\hookrightarrow\hookrightarrow C^1(\overline{\Omega})$ for $p>2$, which yields the control of the $L^\infty$-norm of gradients by the full $W^{2,a}_p(\Omega)$-norm.
\end{proof}

\subsection{Main Weighted Theorem}
\label{subsec:mainweighted}

At this point, we are prepared to state the main result of this subsection. Our aim is to establish the existence of a solution to the Willmore equation in variational form, as reformulated in Lemma~\ref{lem:willrewriten}, for small Dirichlet boundary data. Specifically, we consider the boundary value problem
\begin{align}
\begin{cases}
\Delta^2 u = D_{i} b^{i}_1[u] + D^2_{ij} b^{ij}_2[u] &\ \text{ in }  \ \Omega, \\
u= g_0, \quad  \partial_\nu u =  g_1 &\text{ on } \  \Omega.
\end{cases}    \label{eq:willmoreeq1}
\end{align}
where the nonlinear right-hand side satisfies the structural conditions \eqref{eq:bi} 
\begin{align}
\begin{aligned}
b_1[u] =&\ D^2 u \star D^2 u \star \sum_{k=1}^3 Q^{-2k-1}P_{2k-1}(\nabla u),   \\
    b_2[u]=&\ D^2 u \star \sum_{k=1}^2 Q^{-2k-1}P_{2k}(\nabla u) + D^2 u \star P_2 (\nabla u) \star (Q(1+Q))^{-1}. 
\end{aligned}  \label{eq:willmoreright}
\end{align}

\begin{thm}\label{thm:wlpexistweight}
Assume that $\Omega\subset \R^2$ is a bounded Lipschitz domain with an exterior normal vector field $\nu$. Additionally, suppose
\begin{itemize}
    \item  $p\in(2,\infty), a\in \Big(0,1-\frac{2}{p}\Big)$ with $s:=1-a-\frac{1}{p}$,
    \item $\nu$ satisfies the condition \eqref{eq:condnuA},
    \item $g:=\{g_0,g_1\}\in \dot W^{1+s}_p(\partial\Omega)$ such that $g_0, \nu g_1+\nabla_{\rm tan} g_0\in L^{\infty}(\partial\Omega)$,
    \item $\|g\|_{\dot W^{1+s}_p(\partial\Omega)} <K$ for some $K>0$.
\end{itemize}

Then there exists a constant $\delta=\delta(\Omega,K,p,a)>0$ such that if 
\begin{align*}
    \|\nu g_1+\nabla_{\rm tan} g_0\|_{L^\infty(\Omega)}<\delta,
\end{align*}
then there exists a variational solution $u\in {W^{2,a}_p(\Omega)}\cap C^\infty(\Omega)$ to the Willmore-type Dirichlet problem, thus $u$ solves \eqref{eq:willmoreeq1} with the right-hand side \eqref{eq:willmoreright}.
\end{thm}

Let us briefly discuss the conditions on the parameters $p$ and $a$ in the context of the weighted Sobolev spaces. To ensure ellipticity and the applicability of weighted Sobolev estimates for the right-hand side, we require an upper bound on the $C^1(\overline{\Omega})$-norm of $u$. For this reason, we impose the condition $0\le a\le 1-2/p$ (with the spatial dimension  $n=2$), which allows us to apply the compact embedding $  \mathring{W}^{1,a}_p(\Omega) \hookrightarrow \hookrightarrow  C^{\lambda}(\overline{\Omega})$ from Lemma~\ref{lem:cemb} \cbm c. This, in turn, enables us to control the supremum norm of $\nabla u$.

\begin{proof}
The central idea of the proof is to apply a fixed-point argument based on linearization. To this end, we modify the quasilinear Willmore equation into a linear one by freezing all derivatives except for the biharmonic part. This allows us to formulate a fixed-point problem corresponding to an iterative solution of a linear elliptic differential equation with prescribed Dirichlet data, whose fixed point coincides with the solution of the original Willmore equation.
The iteration is carried out within a closed subset of a weighted Sobolev space  $W^{2,a}_p(\Omega)$. To ensure the validity of this approach, we need to check that the assumptions of the Banach fixed-point theorem are satisfied, which can be achieved by a smallness condition.
The structure of the proof is as follows:
\begin{itemize}
	\item[\cbm1] \emph{Definition of the iteration mapping $G$ as well as the right  iteration set $\M\subset W^{2,a}_p(\Omega)$},
	\item[\cbm2] \emph{The iteration mapping is a self-mapping}: $G\colon \M\to\M$,
	\item[\cbm3] \emph{The iteration mapping is a contraction}: $\exists q\in(0,1) \forall u,w\in\M:$ $\|G(u)-G(w)\|\le q\|u-w\|, $
	\item[\cbm4] \emph{Application of the fixed point theorem} to infer the existence of the fixed point.
\end{itemize}

\noindent\cb{1} \textbf{ Definition of the iteration mapping \& set}\\
We begin by defining the iteration mapping $ G$ by assigning to each $v\in W^{2,a}_p(\Omega)$ the function $ G(v)$ as the  solution $w\in W^{2,a}_p(\Omega)$ to the boundary problem 
\begin{align}
    \begin{cases}
   \Delta^2w =   D_{i} b^{i}_1[v] + D^2_{ij} b^{ij}_2[v] \quad \text{ in }  \ \Omega, \\
    \displaystyle\frac{\partial^k u}{\partial \nu^k} = g_k \quad \text{ on }\ \partial \Omega, \quad 0\le k\le 1.
    \end{cases} \label{wprobitmapw2}
\end{align}
Existence, regularity, and uniqueness are guaranteed by  Proposition~\ref{thm:prop1}.

More precisely, Proposition~\ref{thm:prop1} provides estimates for the $W^{2,p}_a(\Omega)$ and $ L^\infty(\Omega)$ norms of $w$. 
We then define the iteration set by
\begin{align*}
   \M_\delta^K := \left\{ u \in W^{2,a}_p(\Omega) \ \left| \
   \begin{aligned}
     \|\nabla u\|_{L^\infty(\Omega)} &\ \le 2 \Cr{miranda1} \delta, \ 
     \|D^2u\|_{L^1(\Omega;d^{a})}\le 2\Cr{w32021}(a,\Omega) \delta, \\  
     &\ \|u\|_{W^{2,a}_p(\Omega)} \le 2\Cr{lpestimate}(p,a,\Omega) K         
   \end{aligned}\
   \right.\right\},
\end{align*}
with $\delta>0$ representing a smallness parameter, to be determined by several conditions \eqref{eq:wc1}–\eqref{eq:wc5} stated below. By Fatou's lemma 
$\M_\delta^K $ is closed in $W^{2,a}_p(\Omega)$.

Later in the proof, for some $1<q<\infty$, we make use of  the interpolation space $L^q(\Omega; d^{aq})$, lying between the weighted spaces $L^1(\Omega; d^a)$ and $L^p(\Omega; d^{ap})$.  To ensure that the associated Sobolev space $W^{2,a}_q(\Omega)$ is well-defined, we must require 
 $a<1-\frac{2}{q}$. This motivates the choice $q> \frac{2}{1-a}$. At this point, we want to emphasize that $a$ is the same  as in $W^{2,a}_p(\Omega) $!
Since $a<1-\frac{2}{p}$, we set $ q\in\big(\frac{2}{1-a},p\big)$ by $q:= \frac{1}{2}\big(\frac{2}{1-a} + p\big)$ the arithmetic mean. 
Also, we observe that by $q> \frac{2}{1-a}>2$.

Moreover, since $\Omega$ is bounded, we have the compact embedding $W^{2,a}_p(\Omega) \hookrightarrow\hookrightarrow W^{2,a}_q(\Omega)  $ by Hölder's inequality.
In between,  we observe that there exists a power $\gamma:=\frac{p}{q}\frac{q-1}{p-1}\in(0,1)$ such that, like mentioned above, we can interpolate
\begin{align}
    \|D^2u\|_{L^q(\Omega;d^{aq})}\le \|D^2u\|_{L^1(\Omega;d^{a})} ^{1-\gamma}\|D^2u\|_{L^p(\Omega;d^{ap})} ^{\gamma}.\label{eq:int2L}
\end{align}
One of the reasons why we use this interpolation and the estimate for $ \|D^2u\|_{L^1(\Omega;d^{a})}$, is that especially $b_1[w]$  depends linearly on $\nabla w$. In the next step, we have to make sure that
\begin{align*}
    \|b_1[w] \|_{L^{q/2}(\Omega ;d^{aq})}  \le \Cr{miranda1}
     \|D^2 u \|^2_{L^{q}(\Omega ;d^{aq})} \|\nabla u\|_{L^\infty(\Omega)}\le \Cl{loli}(\delta) \delta,
\end{align*}
with some constant $\Cr{loli}$ depending on $\delta $ with $\Cr{loli}(\delta)\to 0$ for $\delta\to 0$ to get $G$ a self-mapping on $M_\delta^K$. Therefore, we estimate $\|D^2 u \|^2_{L^{q}(\Omega ;d^{aq})}\le C(\delta)$ via interpolation.
 
Finally, we introduce the first smallness condition on $\delta$
\begin{align}
   2 \Cr{miranda1}
    \delta_1 = 1, \tag{C1} \label{eq:wc1}
\end{align}
hence for all $\delta\le \delta_1$ we get $\|\nabla u\|_{L^\infty({\Omega})}\le1$ for all $u\in \M_\delta^K$, allowing us to apply  Lemma~\ref{lemm:wprem12} in the sequel.

\noindent\cb{2} \textbf{$G$ is a self-map}\\
In the next step, assume that $w\in \M^K_\delta $. We now aim to prove that $G(w)\in \M^k_\delta$, which is equivalent to verifying the following estimates:
\begin{align*}
     \|\nabla G(w)\|_{L^\infty(\Omega)} &\ \le 2 \Cr{miranda1} \delta, \ 
     \|D^2 G(w)\|_{L^1(\Omega;d^{a})}\le 2\Cr{w32021}(a,\Omega) \delta, \ \|G(w)\|_{W^{2,a}_p(\Omega)} \le 2\Cr{lpestimate}(p,a,\Omega) K. 
\end{align*}
To prepare for the application of Proposition~\ref{weightschauderlpdiv}, we must discuss the admissible values of $t$ for estimating the $b_1$-term. Specifically, we intend to justify the choice  $t=\frac p2$. Therefore, while $0< a< 1-\frac{2}{p}$ it follows that $1<\frac{p(1-a)}2$. We deduce 
 \begin{align}
 1<    \frac{2p}{2-ap+p} = \frac{p}{1+\frac{p(1-a)}{2}} <\frac{p}{2} =t. \label{eq:wpaest}
 \end{align}
By applying Proposition~\ref{weightschauderlpdiv} and Lemma~\ref{lemm:wprem12}, and using the condition $\| \nabla w\|_{L^\infty({\Omega})}\le1$ from $\eqref{eq:wc1}$  we estimate
\begin{align*}
    \|G(w)& 
    \| _{W^{2,a}_p(\Omega)}\\
     &\le 
    \Cr{lpestimate}(p,a,\Omega) 
   \|  
    g  \|_{ \dot W^{1+s}_p (\partial\Omega)} 
     + 
        \Cr{w3202}(p,a,\Omega)  \left(
   \big \|b_1[w] \big\|_{L^{p/2}(\Omega ;d^{ap})}
   +  \big\|b_2[w] \big\|_{L^{p}(\Omega ;d^{ap})}
    \right)
    \\
    &\le \Cr{lpestimate}
    K
    +
    \C(p,a,\Omega) \left( \big\| D^2w\big\|_{L^p(\Omega;d^{ap})}^2 \big\|\nabla w\|_{L^\infty(\Omega)}  
   + \big\| D^2w\big\|_{L^p(\Omega;d^{ap})} \big\|\nabla w\|_{L^\infty(\Omega)}      \right)\\
     &\le \Cr{lpestimate}   K  + \Cl{w3003}(p,a,\Omega)( K^2 \delta +  K \delta ).
 \end{align*}
 Therefore, we impose the second constraint on $\delta$ by choosing  $\delta_2$ such that 
 \begin{align}
    \Cr{w3003}(p, a, \Omega)( K^2 \delta_2 +K \delta_2 )= \Cr{lpestimate}(p,a,\Omega) K. \label{eq:wc2} \tag{C2} 
 \end{align}
From now on, we restrict to $\delta\in (0,\delta_2)$. This guarantees that $ \|G(w)\| _{W^{2,a}_p(\Omega)} \le 2\Cr{lpestimate}(p,a,\Omega)K$, thus verifying the third condition for membership in $\M_\delta^K$.

 For the second condition in the definition of $\M_\delta^K$, we again apply Proposition~\ref{weightschauderlpdiv}, this time using the estimate in the space  $W^{2,a}_q(\Omega)$, i.e., replacing $p$ by $q$ and setting $t=q/2$. Note that the required condition ${2q}/({2-aq+q}) <{q}/{2}$ is also satisfied. The corresponding estimate for   ${L^1(\Omega;d^a)} $ norm of $D^2G(w)$ then reads:
\begin{align*}
    \|D^2G(w)&\| _{L^1(\Omega;d^a)} \\
    &\le
     \Cr{w32021}(a,\Omega) 
   \| \nu g_1
   +
   \nabla_{\rm tan} g_0\|_{L^\infty(\partial\Omega)}
     \\
     &\ \qquad + 
        \Cr{w3202}(q, a, \Omega) \left(
    \big\|b_1[w]\big\|_{L^{q/2}(\Omega ;d^{2aq/2})}
   + \big\|b_2[w]\|_{L^{q}(\Omega ;d^{aq})}
    \right)
    \\
    &\le  
    \Cr{w32021}(a,\Omega)  
    \delta
    +
        \C\left( \big\| D^2w\big\|_{L^q(\Omega;d^{aq})}^2 \big\|\nabla w\|_{L^\infty(\Omega)}  
   + \big\| D^2w\big\|_{L^q(\Omega;d^{aq})} \big\|\nabla w\|_{L^\infty(\Omega)}      \right)\\
     &\overset{\eqref{eq:int2L}}{\le}\Cr{w32021}(a,\Omega)  
    \delta  + \Cl{wt03}(p,a,\Omega)\big((\delta^{1-\gamma}K^\gamma)^2 + \delta^{1-\gamma}K^\gamma \big)\delta.
 \end{align*}
 since $q$ and $\gamma=\frac{p}{q}\frac{q-1}{p-1}$ depends on $a $ and $p$. We now impose the third smallness condition by choosing $\delta_3$  such that 
 \begin{align}
     \Cr{wt03}(p,a,\Omega)\big((\delta_3^{1-\gamma}K^{\gamma})^2 + \delta_3^{1-\gamma}K^\gamma \big)
     =
     \Cr{w32021}(a,\Omega) . \tag{C3} \label{eq:wc3}
 \end{align}
From this point on, we restrict to $\delta\le \delta_3$, which ensures that $ \|G(w)\| _{W^{2,a}_p(\Omega)} \le 2 \Cr{w32021}(q,a,\Omega) \delta$, i.e., the second condition for membership in $\M_\delta^K$ is satisfied.
 
It remains to verify the gradient estimate. Again, by Proposition~\ref{weightschauderlpdiv} and Lemma~\ref{lemm:wprem12} it follows
 \begin{align*}
             \|\nabla G(w)\|_{L^\infty(\Omega)} &\le \Cr{miranda1}
   \| \nu g_1
   +
   \nabla_{\rm tan} g_0\|_{L^\infty(\partial\Omega)}
\\
&\ \quad  + 
        \Cr{w3202}(q,a,\Omega) \left(
     \big\|b_1[w] \big\|_{L^{q/2}(\Omega ;d^{aq})}
   +  \big\|b_2[w] \big\|_{L^{q}(\Omega ;d^{aq})}
    \right) \\
   &\le \Cr{miranda1} \delta  + \Cr{wt03}(q,a, \Omega)\big((\delta^{1-\gamma}K^\gamma)^2 + \delta^{1-\gamma}K^\gamma \big)\delta.
 \end{align*}
Therefore, we set the fourth condition by choosing $\delta_4$  such that 
\begin{align}
         \Cr{wt03}(q,a, \Omega)\big((\delta_4^{1-\gamma}K^\gamma)^2 + \delta_4^{1-\gamma}K^\gamma \big)=\Cr{miranda1}  . \tag{C4}\label{eq:wc4}
\end{align}
Combining the results from the previous steps, we conclude that the mapping $G\colon\M_\delta^K \to \M_\delta^K $ is a self map for all $\delta$ smaller than $\delta_1$, $\delta_2, \delta_3 $ and $ \delta_4$.

\noindent\cb{3} \textbf{ $G$ is a contraction}\\
The final property we want to verify is the contraction property. That is, we aim to find a constant  $q\in(0,1)$ such that for all $u,w\in \M_\delta^K $,  the following inequality holds: 
\[
 \|G (u)-G (w)\|_{W^{2,a}_p(\Omega)}  \le q \|u-w\|_{{W^{2,a}_p(\Omega)}}.
 \] We first observe that the difference $G(u)-G(w)$ is a variational solution of the boundary value problem
 \begin{align*}
\Delta^2\big (G (u)- G (w) \big) = D_i \big( b^i_1[u]- b^i_1[w] \big) + D^2_{ij} \big( b^{ij}_2[u] - b^{ij}_2[w] \big)  \quad \text{ in } \ \Omega, 
\end{align*}
in the class $\mathring W^{2,a}_p(\Omega)$. 

As in the previous step, we apply Proposition~\ref{weightschauderlpdiv} in combination with Lemma~\ref{lemm:wprem12}, using the bounds  $\| \nabla u\|_{L^\infty({\Omega})},\| \nabla w\|_{L^\infty({\Omega})}\le 1$  ensured by condition \eqref{eq:wc1}. 
In preparation, we recall from \eqref{eq:wpaest} that we may choose a parameter $t>0$ such that $ 1< \frac{2p}{2-ap+p}<t< \frac{p}{2}$ as required for Proposition~\ref{weightschauderlpdiv}. In particular, we select the arithmetic mean $t:=\frac{p}{2-ap+p}+ \frac{p}{4}$. The reason for choosing a different value of $t$ compared to the previous step is that, in Lemma~\ref{lemm:wprem12}, the right-hand side of the estimate \eqref{eq:westvorlp2a} does not involve $L^\infty(\Omega)$-norms of $\nabla u$ or $ \nabla w$, previously allowed us to exploit smallness. 

With this setup, we obtain the estimate:
\begin{align*}
    \|G (u)-G (w)\|_{W^{2,a}_p(\Omega)} &\ 
     \le 
             \Cr{w3202}(p,t,\Omega) \left(
    \|b_1[u]-b_1[w]\|_{L^{t}(\Omega ;d^{2at})}
   + \|b_2[u]-b_2[w]\|_{L^{p}(\Omega ;d^{ap})}
    \right).
\end{align*}
At this point, we estimate the two terms on the right-hand side separately. Starting with the $b_1$-term, we use estimate \eqref{eq:westvorlp2a} to obtain
\begin{align*}
\big\|b_1[u] &\ -b_1[w] \big\|_{L^{t}(\Omega;d^{2at})} 
\le \Cr{whoelder1a}   \| u-w \|_{W^{2,a}_{2t}(\Omega)} \Big(  \| D^2u \|_{L^{2t}(\Omega;d^{2at})}^2 + \| D^2w \|_{L^{2t}(\Omega;d^{2at})}^2 \Big).
\end{align*}
The terms $\| D^2u \|_{L^{2t}(\Omega;d^{2at})}^2 $ and $\| D^2w \|_{L^{2t}(\Omega;d^{2at})}^2 $ could not be made small enough if we had chosen $t=p/2$ as in the previous step. We therefore proceed by applying an  $L^p$-interpolation argument
\begin{align*}
\big\|b_1[u] &\ -b_1[w] \big\|_{L^{t}(\Omega;d^{2at})} \\
\le &\
\C \| u-w \|_{W^{2,a}_{2t}(\Omega)}\Big( \| D^2u \|_{L^{1}(\Omega;d^{a})}^{2\alpha} \| D^2u \|_{L^{p}(\Omega;d^{pa})} ^{2(1-\alpha)}+ \| D^2u \|_{L^{1}(\Omega;d^{a})}^{2\alpha} \| D^2u \|_{L^{p}(\Omega;d^{pa})} ^{2(1-\alpha)} \Big),   
\end{align*}
where  $\alpha\in(0,1)$  is chosen such that  $2t=\alpha  +(1-\alpha)p$ lies strictly between $1$ and $p$. Since $\Omega\subset\R^2$ is a bounded domain, we have the compact embedding $\mathring W^{2,a}_p(\Omega)\hookrightarrow\hookrightarrow \mathring W^{2,a}_{2t}(\Omega)$, which yields
\begin{align*}
\big\|b_1[u]-b_1[w] \big\|_{L^{t}(\Omega;d^{2at})}
\le &\ \Cl{wzzz} \| u-w \|_{W^{2,a}_p(\Omega)}  (\delta^{1-\gamma}K^\gamma)^{2\alpha} K ^{2(1-\alpha)} .
\end{align*}
Furthermore, let us estimate the $b_2$-terms using \eqref{eq:westvor21a}
\begin{align*}
 \big\| b_2[u]\ - &\ b_2[w]\big\|_{L^{p}(\Omega;d^{ap})} \\
 \ \le &\ \Cr{whoelder1a}    \|u-w \|_{W^{2,a}_p({\Omega})} 
       \left(
 \|\nabla u\| _{L^\infty({\Omega})} +\|\nabla w\| _{L^\infty({\Omega})} 
               \right)  
               \left(
 \| u\| _{W^{2,a}_p(\Omega)} +\|w\| _{W^{2,a}_p(\Omega)} 
               \right)\\
    \ \le &\ \C  \|u-w \|_{W^{2,a}_p({\Omega})} 
 \delta K.        
\end{align*}
Combining the two estimates above, we obtain for some constant $\Cl{w3schritt}= \Cr{w3schritt}(p,a,\Omega)$  the inequality
\begin{align*}
     \|G (u)-G (w)\|_{W^{2,a}_p(\Omega)}\le \Cr{w3schritt}\big( (\delta^{1-\gamma}K^\gamma)^{2\alpha} K ^{2(1-\alpha)}+ \delta K \big)   \|u-w\|_{W^{2,a}_p(\Omega)}.
\end{align*}
We now impose the fifth and final smallness condition by choosing $\delta_5=\delta_5(a,p,K, \Omega)$  such that
\begin{align}
    \Cr{w3schritt}\big( (\delta_5^{1-\gamma}K^\gamma)^{2\alpha} K ^{2(1-\alpha)} + \delta_5 K \big)   = \frac{1}{2}. \tag{C5} \label{eq:wc5}
\end{align}
Then, for any $\delta \le  \delta_5$ satisfying all the previous constraints \eqref{eq:wc1}, \eqref{eq:wc2}, \eqref{eq:wc3}, \eqref{eq:wc4}  we obtain for all $u,w\in \M^K_\delta$ the contraction estimate:
\begin{align*}
        \|G (u)-G (w)\|_{W^{2,a}_p(\Omega)} &\ \le \frac{1}{2} \|u-w\|_{{W^{2,a}_p(\Omega)}}.
\end{align*}
In conclusion, by choosing $\delta\le \min\{\delta_1, \delta_2, \delta_3, \delta_4,\delta_5\}$,
the mapping $G\colon \M_\delta^K\to \M_\delta^K$ becomes a contraction on $\M_\delta^K $.

\noindent\cb{4} \textbf{ Application of the fixed-point theorem}\\
Finally, we combine all conditions \eqref{eq:wc1},\eqref{eq:wc2},\eqref{eq:wc3},\eqref{eq:wc4} and \eqref{eq:wc5} on the parameter $\delta$ and define
\begin{align}
    0<\delta = \delta_0:= \min\Big\{\delta_1, \delta_2, \delta_3, \delta_4,\delta_5 \Big\}. \label{eq:weightedlpdcond}
\end{align}
By the Banach fixed-point theorem, we obtain the existence of a unique fixed point  $u^\ast\in\M^K_\delta \subset{W^{2,a}_p({\Omega})}$ satisfying
$
	u^\ast=G u^\ast.
$
Therefore,
$u^\ast\in W^{2,a}_p(\Omega)$ is  a variational solution of the Willmore equation in the space $W^{2,a}_p(\Omega)$. 
The interior regularity of $u^\ast$ is established via a bootstrapping argument, as detailed in \cite[Subsection 5.5]{gulyak2024boundary}.
\end{proof}

\subsection{Pure Lipschitz Boundary Case}
\label{subsec:purelipsch} 
As mentioned earlier in Subsection~\ref{eq:weightedsobolev}, there exists a trade-off between the regularity of the boundary $\partial\Omega$ and the regularity of the Dirichlet boundary data. If one wishes to relax the regularity assumptions on $\partial\Omega$, requiring only that it be a Lipschitz boundary and omitting the abstract condition \eqref{eq:condnuA}, this comes at the cost of stricter requirements, represented by constant $\Cr{2vmobound}$ in the following theorem, on the parameters $a$ and $p$  in the weighted Sobolev spaces, as well as on the regularity of the boundary data.
\setcounter{thm}{1}
\begin{thm}
Assume that $\Omega\subset \R^2$ is a bounded Lipschitz domain whose Lipschitz constant
does not exceed $M$. Then there exist a constant $\Cr{2vmobound}=\Cr{2vmobound}(M)>0$ such that if
\begin{itemize}
    \item  $p\in(2,\Cr{2vmobound}), a\in \Big(0,1-\frac{2}{p}\Big)\cap(0,\Cr{2vmobound})$ with $s:=1-a-\frac{1}{p}$,
    \item $g:=\{g_0,g_1\}\in \dot W^{1+s}_p(\partial\Omega)$ such that $g_0, \nu g_1+\nabla_{\rm tan} g_0\in L^{\infty}(\partial\Omega)$,
    \item $\|g\|_{\dot W^{1+s}_p(\partial\Omega)} <K$ for some $K>0$.
\end{itemize}

Then there exists a constant $\delta=\delta(\Omega,K,p,a,M)>0$ such that if 
\begin{align*}
    \|\nu g_1+\nabla_{\rm tan} g_0\|_{L^\infty(\Omega)}<\delta
\end{align*}
then there exists a variational solution $u\in {W^{2,a}_p(\Omega)}\cap C^\infty(\Omega)$ to the Willmore-type Dirichlet problem, thus $u$ solves \eqref{eq:Willmoreellipt} with the right-hand side \eqref{eq:bi}.
\end{thm}
\begin{proof} See Remark \ref{remark:lipsch}. In this proof, we apply \cite[Theorem 2, p. 44 ]{maz2011recent} in place of Proposition~\ref{thm:prop1}, while employing the same techniques as in the proof of Theorem~\ref{thm:wlpexistweight}.
\end{proof}

\subsection{Hölder Boundary Data}
\label{subsec:hoelderboundary}
In this subsection, we establish existence results for boundary value problems with Hölder continuous boundary data. The key observation is that the Hölder spaces $C^{\alpha}(\partial\Omega)$ are  continuously embedded into certain Besov spaces, which in turn characterize the Dirichlet boundary data spaces $ \dot W^{1+s }_p(\partial\Omega)$ for $s<\alpha$. This embedding allows us to solve the corresponding weighted Sobolev boundary value problem, since for any fixed $\alpha$ one can choose some $p>2$ such that the required regularity is satisfied.
\setcounter{thm}{0}
\begin{thm}
Let $\Omega\subset\R^2$ be a bounded domain with $\partial\Omega\in C^{1+\alpha}$ for some $\alpha\in(0,1)$. Assume that $\beta\in(0,\alpha), g_0\in C^{1+\alpha}(\partial\Omega)$ and $g_1\in C^{\alpha}(\partial\Omega)$. Additionally, we suppose that $\|g_0'\|_{C^{\alpha}(\partial\Omega)}+ \|g_1\|_{C^{\alpha}(\partial\Omega)} <K$ for some $K>0$. 

Then there exists a constant $\delta=\delta(\alpha,\beta, K, \Omega)>0$ such that if $\|g_0'\|_{C^0(\partial\Omega)}+\|g_1\|_{C^{0}(\partial\Omega)}<\delta$, then there exists a solution $u\in C^{1+\beta}(\overline\Omega)\cap C^\infty(\Omega)$ to the Willmore Dirichlet problem \eqref{eq:dirwillmore}.

\end{thm}
\begin{proof} 
Let $s\in(\beta,\alpha) $ and $p>2$  where the precise choice of $p$ will be specified later in \eqref{eq:spparam} depending on $\alpha$ and $\beta$.
We begin by showing that $C^{\alpha}(\partial\Omega)\hookrightarrow \hookrightarrow B^{s}_p(\partial\Omega)$ for each $p>2$. Let $f\in C^{\alpha}(\partial{\Omega})$. We recall the definition of the Besov norm:
\begin{align*}
 \| f\| _{B^s_p(\partial\Omega)}&\ = \| f\|_{L^p(\partial\Omega)} 
 + \left(\int_{\partial\Omega}\int_{\partial\Omega} \frac{\big| f(x)-f(y) \big|^p}{|x-y|^{2-1+s p}} \diff S_x \diff S_y \right)^\frac{1}{p} \\
 &\ \le | \partial\Omega| \|f\|_{C^0(\partial\Omega)} + 
    \left(\int_{\partial\Omega}\int_{\partial\Omega} |x-y|^{-1+(\alpha-s )p} \diff S_x \diff S_y \right)^\frac{1}{p} \|f \|_{C^\alpha(\partial\Omega)}< \infty,
\end{align*} 
since $\alpha-s > 0$ and $\partial\Omega\in C^{1+\alpha}$. Furthermore, the condition $\nu\in C^{\alpha}(\partial\Omega)$ yields that $g'_0=\nabla_{\rm tan} g_0\in C^\alpha(\partial\Omega)$ and moreover $g_0\in L^1_p(\partial\Omega)$. Together with $g_1\in C^\alpha(\partial\Omega)$ it follows that $\nu g_1+\nabla_{\rm tan} g_0\in B^s _p(\partial\Omega)$ and we estimate 
\begin{align*}
    \|\nu g_1+\nabla_{\rm tan} g_0\|_{L^\infty(\partial\Omega)} \le \C(\partial\Omega)\big( \| g_0\|_{C^1(\partial\Omega)} + \|g_1\|_{C^0(\partial\Omega)} \big),
\end{align*}
therefore we conclude that the boundary data $g:=\{g_0,g_1\}$ belong to the trace space $ \dot W^{1+s }_p(\partial\Omega)$.

In order to apply Theorem~\ref{thm:wlpexistweight}, we must ensure that for a given $s$ the parameters $p$ and $a$ satisfy the conditions $p>2$ and $0\le a=1-s -\frac{1}{p}<1-\frac{2}{p}$. This leads to the requirement
\begin{align} \label{eq:condp}
    \frac{1}{p}< \min\left(s , 1-s\right).
\end{align}
Under this condition, together with the regularity assumptions on the boundary data, Theorem~\ref{thm:wlpexistweight} yields the existence of a solution $u\in W^{2,a}_p(\Omega)$ to the Willmore problem \eqref{eq:Willmoreellipt} with the right-hand side \eqref{eq:bi}. In particular, it follows that $u\in W^{2,2}_{\text{loc}}(\Omega)$.

Our next goal is to show that $u\in C^{1+\beta}(\overline{\Omega})$. Therefore, we  require a suitable embedding from weighted Sobolev spaces into Hölder spaces. By Lemma~\ref{lem:cemb} (c), with
$
    \lambda= 1-a-\frac{2}{p} = s-\frac{1}{p}
$ we obtain the compact embedding
\begin{align*}
    \mathring{W}^{2,a}_p(\Omega) \hookrightarrow\hookrightarrow C^{1+s-\frac{1}{p}}(\overline{\Omega}).
\end{align*}
To exploit this embedding, we seek a function $w\in C^{1+\beta}(\Omega)$ such that the difference  $w-u $ lies the homogeneous space $\mathring{W}^{2,a}_p(\Omega) $. 
Thus, let $w\in W^{2,a}_p(\Omega)$ be the solution to the biharmonic boundary value problem:
\begin{align*}
\begin{cases}
\Delta^2 w =0, \quad \text{ in } \ \Omega,\\
w= g_0, \quad  \partial_\nu u =  g_1 \quad \text{ on }\  \partial \Omega.
\end{cases}    
\end{align*}
Then, by \cite[Theorem 2.19, p. 45]{gazzola2010polyharmonic}, we obtain the following Schauder estimate:
\begin{align*}
\|w\|_{C^{1+\alpha}(\partial\Omega)}\le  \C\|g_0\|_{C^{1+\alpha}(\partial\Omega)} +
\C\|g_1\|_{C^{\alpha}(\partial\Omega)}.
\end{align*}

Finally, we specify the parameters $s$ and $p$ such that all necessary conditions are fully satisfied, depending on the Hölder powers $\alpha$ and $\beta$
\begin{align}
    s:= \frac{1}{2}(\alpha+\beta) \quad \text{ and } \quad p:= \frac{1}{s-\beta}=\frac 2{\alpha-\beta}. \label{eq:spparam}
\end{align}
Then we verify that the condition on $p$ from \eqref{eq:condp} is satisfied: 
\begin{align*}
    \frac{1}{p} =s -\beta < s\quad  \text{ and }
    \quad\frac{1}{p} = \frac{1}{2}(\alpha-\beta) = \alpha -s < 1-s.
\end{align*}
With this choice, and recalling that $w-u\in \mathring{W}^{2,a}_p(\Omega)  $, the compact embedding yields
 $w-u \in C^{1+\beta}(\overline{\Omega})$ and
\begin{align*}
    \|w-u\|_{C^{1+\beta}(\overline{\Omega})}\le \C\|w-u\|_{\mathring{W}^{2,a}_p(\Omega) }.
\end{align*}
Consequently, we conclude that $u\in  C^{1+\beta}(\overline{\Omega})\cap C^\infty(\Omega)$ using interior regularity. Lastly,  we observe that the graphical Willmore equation is invariant under horizontal translations. Therefore, the term $\|g_0\|_{C^{1+\alpha}(\partial\Omega)}$ in the estimate can be replaced by $\|g_0'\|_{C^{\alpha}(\partial\Omega)}$.
\end{proof}




{\bf Conflict of Interest:} The authors declared that they have no conflicts of interest to this work.

\

{\bf Availability of data and material:} Data sharing not applicable to this article as no datasets were generated or analyzed during the current study.

\

{\bf Acknowlegments:} The research was carried out during the author’s PhD under the supervision of Hans-Christoph Grunau. The author is very grateful for his guidance, careful reading, corrections, and helpful discussions.


\bibliographystyle{myalpha}
\bibliography{Willmore_Elliptic_Paper}


\end{document}